\documentclass[oneside]{amsart}

\usepackage[letterpaper,body={14.0cm,22.0cm}, mag=1000]{geometry}
\usepackage{amssymb}
\usepackage{amsthm}
\usepackage{amscd}
\usepackage{enumitem}
\usepackage{float}
\usepackage{placeins}
\usepackage{caption}

\numberwithin{equation}{section}
\theoremstyle{plain}

\newtheorem{cor}[equation]{Corollary}
\newtheorem{lemma}[equation]{Lemma}

\newtheorem{proposition}[equation]{Proposition}

\newtheorem{thm}[equation]{Theorem}

\newtheorem{qns}[equation]{Question}

\theoremstyle{definition}
\newtheorem{definition}[equation]{Definition}
\newtheorem{example}[equation]{Example}

\newcommand{\dlabel}[1]{\ifmmode \text{\ttfamily \upshape [#1] } \else
{\ttfamily \upshape [#1] }\fi \label{#1}}

\renewcommand{\det}{\operatorname{det} }
\newcommand{\gen}[1]{\left < #1 \right >}
\newcommand{\Aut}{\operatorname{Aut} }

\newcommand{\im}{\operatorname{Im} }
\newcommand{\Ker}{\operatorname{Ker} }

\newcommand{\Hol}{\operatorname{Hol} }

\begin{document}
\setcounter{page}{1}
\title[On $\lambda$-homomorphic skew braces]
{On $\lambda$-homomorphic skew braces}

\author{Valeriy G. Bardakov}
\address{Sobolev Institute of Mathematics, 4 Acad. Koptyug avenue, 630090, Novosibirsk, Russia.}
\address{Novosibirsk State  University, 2 Pirogova Street, 630090, Novosibirsk, Russia.}
\address{Novosibirsk State Agrarian University, Dobrolyubova street, 160, Novosibirsk, 630039, Russia.}
\address{Regional Scientific and Educational Mathematical Center of Tomsk State University, 36 Lenin Ave., Tomsk, Russia.}
\email{bardakov@math.nsc.ru}

\author{Mikhail V. Neshchadim}
\address{Sobolev Institute of Mathematics, 4 Acad. Koptyug avenue, 630090, Novosibirsk, Russia.}
\address{Novosibirsk State  University, 2 Pirogova Street, 630090, Novosibirsk, Russia.}
\address{Novosibirsk State Agrarian University, Dobrolyubova street, 160, Novosibirsk, 630039, Russia.}
\address{Regional Scientific and Educational Mathematical Center of Tomsk State University, 36 Lenin Ave., Tomsk, Russia}
\email{neshch@math.nsc.ru}

\author{Manoj K. Yadav}
\address{Harish-Chandra Research Institute, HBNI, Chhatnag Road, Jhunsi, Allahabad-211 019, India}
\email{myadav@hri.res.in}


\begin{abstract}
For a skew left brace $(G, \cdot, \circ)$, the map $\lambda : (G, \circ) \to \Aut (G, \cdot),~~a \mapsto \lambda_a,$
where $\lambda_a(b) = a^{-1} \cdot (a \circ b)$ for all $a, b \in G$,
is a group homomorphism. Then $\lambda$ can also be viewed as a map from $(G, \cdot)$ to $\Aut (G, \cdot)$, which, in general, may not be a homomorphism. We study skew left braces $(G, \cdot, \circ)$ for which $\lambda : (G, \cdot) \to \Aut (G, \cdot)$ is a homomorphism. Such skew left braces will be called $\lambda$-homomorphic.   We formulate necessary and sufficient conditions under which a given homomorphism $\lambda :  (G, \cdot) \to \Aut (G, \cdot)$ gives rise to  a skew left brace, which, indeed, will be $\lambda$-homomorphic. As an application, we construct skew left braces when $(G, \cdot)$ is either a free group or a free abelian group. We prove that any $\lambda$-homomorphic skew left brace is an extension of a trivial skew brace by a trivial skew brace. Special emphasis is given on $\lambda$-homomorphic skew left brace for which the image of $\lambda$ is cyclic. A complete characterization of such skew left braces on the free abelian group of rank two is obtained.

 \end{abstract}
\maketitle

\section{Introduction}

A  triple $(G, \cdot, \circ)$, where $(G, \cdot)$ and $(G, \circ)$ are  groups,    is said to be a \emph{skew left brace} if
 \begin{equation}
 g_1 \circ (g_2 \cdot g_3) =  (g_1 \circ g_2) \cdot g_1^{-1} \cdot  (g_1 \circ g_3)
 \end{equation}
 for all $g_1, g_2, g_3 \in G$, where $ g_1^{-1}$ denotes the  inverse of $g_1$ in $(G, \cdot)$. We call  $(G, \cdot)$ the \emph{additive group} and $(G, \circ)$ the \emph{multiplicative  group} of the skew left brace $(G, \cdot, \circ)$. A skew left brace $(G, \cdot, \circ)$ is said to be a \emph{left brace} if $(G, \cdot)$ is an abelian group.  In this article we mainly consider left braces and  skew left braces. So, we'll mostly suppress the word `left' and only say brace(s) or skew brace(s).

  The concept of  braces was introduced by Rump \cite{R2007} in 2007 in connection with non-degenerate involutive set theoretic solutions of the quantum Yang-Baxter equation. Thereafter the subject received a tremendous attention of the mathematical community; see \cite{BCJO18, FC2018, WR2019, AS2018} and the references therein.  Interest in the study of set theoretic solutions of the quantum Yang-Baxter equations was intrigued by the paper \cite{D1992}  of Drinfeld, published in 1992.  The concept of skew  braces was introduced by Guarnieri and Vendramin \cite{GV2017} in 2017 in connection with non-involutive non-degenerate set theoretic solutions of the quantum Yang-Baxter equation. For more on skew braces see \cite{CSV19, JKAV2019, SV}. A detailed survey on problems on skew braces is presented in \cite{LV2019}.

For a skew left brace $(G, \cdot, \circ)$, it was  proved in \cite{GV2017} that the map
$$
\lambda  :  (G, \circ) \to \Aut (G, \cdot),~~a \mapsto \lambda_a
$$
is a group homomorphism, where $\Aut (G, \cdot)$ denotes the automorphism  group of $(G, \cdot)$ and $\lambda_a$ is given by $\lambda_a(b) = a^{-1} \cdot (a \circ b)$ for all $a, b \in G$.
Obviously, $\lambda$ can also be viewed as a map from $(G, \cdot)$ to $\Aut (G, \cdot)$, which, in general, may not be a homomorphism. 
We say that  $(G, \cdot, \circ)$ is a $\lambda$-{\it homomorphic skew left brace} if the map $\lambda : (G, \cdot) \to \Aut (G, \cdot)$ is a homomorphism. A  $\lambda$-homomorphic skew left brace $(G, \cdot, \circ)$ is said to be a $\lambda$-{\it cyclic skew left brace} if the image $\im \lambda$ is a cyclic subgroup of $\Aut (G, \cdot)$.  
$\lambda$-homomorphic skew braces of finite order were studied in \cite{CC}.

In the present article our main focus is on the study of the following problem: For a given group $(G, \cdot)$, define a binary operation `$\circ$' such that $(G, \cdot, \circ)$ is a non-trivial skew  brace, that is,  `$\cdot$' and `$\circ$' do not coincide.
We investigate $\lambda$-homomorphic and $\lambda$-cyclic skew braces (mainly of infinite order) and show  that any $\lambda$-homomorphic skew brace is metatrivial, that is, an extension of a trivial skew brace by a trivial skew brace. 

Description of all skew braces of given order is much more complicated than description of all groups of given order. For example, there are 51 groups of order 32, but there are 1223061  skew   braces and 25281 braces  of order 32 (see \cite{BNY}). It is very  difficult  to describe all skew braces even with a given additive group.  So the construction of certain specific type of skew braces is highly desirable. Among other things, we construct various  skew braces of infinite order, mainly on free groups and free abelian groups. 

The paper is organised as follows. In Section \ref{lambda} we recall a theorem from \cite{GV2017} which provides a connection between skew braces and regular subgroups. Theorem \ref{t1} gives necessary and sufficient conditions under which the homomorphism $\lambda : (G, \cdot) \to \Aut (G, \cdot)$ gives rise to a regular subgroup $H_{\lambda}$ of the holomorph of $(G, \cdot)$, and hence defines a $\lambda$-homomorphic skew  brace. As a consequence, it follows that if the group $(G, \cdot)$ is non trivial and the kernel of  $\lambda$ is trivial, then $H_{\lambda}$ can not be a group. On the other side, we give an example  of a skew brace for which $\lambda$ is anti-homomorphism with trivial kernel. We find the structure of $\lambda$-homomorphic skew braces and prove that any  $\lambda$-homomorphic skew brace is an  extension of a trivial skew l brace by a trivial skew  brace. At the end we construct $\lambda$-homomorphic skew braces on free and free abelian groups of finite ranks.

 In Section \ref{lambda-cyc} we study $\lambda$-cyclic skew braces.  We introduce homogeneous presentations and homogeneous groups. Homogeneous groups include, in particular, free goups,  free solvable groups,  free nilpotent groups, groups of knots and so on. In Theorem \ref{lcb} we prove that if $(G, \cdot)$ is a homogeneous group, then a homomorphism $\lambda : (G, \cdot) \to \Aut (G, \cdot)$ with cyclic image defines a regular subgroup and hence a $\lambda$-cyclic skew  brace. Using this theorem we construct skew braces on free group.

 In Section \ref{cyc-ab} we construct $\lambda$-cyclic  braces on free abelian groups. In particular, we prove that on the free abelian group of rank 2, there are only three non-isomorphic $\lambda$-cyclic braces. We also construct a brace on any free abelian group  of finite rank.

 In Section \ref{sym} we take up the concept of symmetric skew braces, which was  introduced by  Childs \cite{Chi} (under the name bi-skew brace) and studied by Caranti \cite{Ca}. We give necessary and sufficient conditions under which a skew brace is  symmetric. We give examples of symmetric braces on free abelian groups.  Theorem \ref{t6.6} says  that every $\lambda$-cyclic skew brace is symmetric. We construct a skew brace by exact factorization of free product of two groups, and prove that the  skew brace constructed on the  free product $B * C$ is $\lambda$-homomorphic  if $B$ is abelian. Example of a skew brace with finitely generated additive group and infinitely generated multiplicative group is presented.  An example of skew brace with non-nilpotent metabelian  additive group and infinitely generated free abelian multiplicative group is also presented. At the end of this section we establish the existence of  a non-trivial skew brace whose additive group is the group of a tame knot.

It is well known that if $(A, +, \circ)$ is a finite brace, then the multiplicative group $(A, \circ)$ is solvable. Nasybullov \cite{Nas} proved that this is not true in general.  Moreover, a two sided brace was constructed in \cite{CSV19} whose multiplicative group  contains a non-abelian free group.   In Section \ref{const}  we present a similar example  and prove (Proposition \ref{p7.1}) that there is a two-sided brace $(A, +, \circ)$ whose multiplicative group $(A,  \circ)$ contains a non-abelian free subgroup. Finally we prove that if  $(A, + )$ is an abelian group,  $B$  its subgroup of index  2, then there is a brace $(A, + , \circ)$ in  which   $(A, \circ)$ is isomorphic to the semi-direct product $B \rtimes \mathbb{Z}_2$. 

We formulate some open problems,  which naturally arise during our study. 

When the additive group of a skew brace is non-abelian, we'll mostly suppress `$\cdot$' and use $ab$ for $a \cdot b$. For  braces,  we denote the additive operation `$\cdot$' by `$+$'. For a skew brace $(G, \cdot, \circ)$, the inverse of $a \in G$ with respect to `$\circ$' is denoted by $\bar a$.


\section{$\lambda$-homomorphic skew braces} \label{lambda}

In this section we provide a structure of  $\lambda$-homomorphic skew brace and construct various such skew braces.

Let $(G, \cdot, \circ)$ be a skew brace. Then, as proved in \cite{GV2017}, the map
$$
\lambda : (G, \circ) \to \Aut (G, \cdot),~~a \mapsto \lambda_a,
$$
where $\lambda_a(b) = a^{-1} (a \circ b)$, is a group homomorphism.  The inverse $\bar a$ of $a \in G$ with respect to `$\circ$'  is given by $\lambda_a^{-1} (a^{-1})$.

Let $G$ be a group. The holomorph of $G$ is the group $\Hol G := \Aut  G \ltimes G$, in which the product is given by
$$
(f,a)(g,b)=(fg,af(b))
$$
for all $a,b \in G$ and $f,g \in \Aut G$.
Any subgroup $H$ of $\Hol  G$ acts on $G$ as follows
$$
(f,a) \cdot b = af(b),\quad a,b \in G,\,\, f \in \Aut G.
$$
A subgroup $H$ of $\Hol  G$ is said to be {\it regular} if for each $a \in A$
there exists a unique $(f,x) \in H$ such that $xf(a)=1$. It is  equivalent to the fact that
the action of $H$ on $G$ is free and transitive. Let $\pi_2 : \Hol G \to G$ denote the projection map from $\Hol G$ onto the second component $G$.

The following theorem from \cite[Theorem 4.2]{GV2017} provides a connection between skew braces and regular subgroups.

\begin{thm}\label{gv2017}
Let $(G, \cdot, \circ)$ be skew  brace. Then $\left\{ \, (\lambda_a, a)  \mid  a \in G  \,  \right\}$
is a regular subgroup of $\Hol G$, where $\lambda_a(b) = a^{-1}(a \circ b)$ for all $b \in A$.

Conversely, if $(G, \cdot)$ is a group and $H$ is a regular subgroup of $\Hol (G, \cdot)$, then $(G, \cdot, \circ)$ is a skew  brace such that $(G, \circ) \cong H$, where $a \circ b=a f(b)$
with $(\pi_2|_H)^{-1}(a)=(f, a)\in H$.
\end{thm}

\begin{definition}
A skew brace $(G, \cdot, \circ)$ is said to be a $\lambda$-{\it homomorphic skew brace} if the map $\lambda : (G, \cdot) \to \Aut (G, \cdot)$, defined  by  $\lambda_a(b) = a^{-1} (a \circ b)$, $a, b \in (G,\cdot)$, is a homomorphism.
\end{definition}

For  a  group  $G = (G, \cdot)$, we desire to define a homomorphism $\lambda : G \to \Aut G$ such that $(G, \cdot, \circ)$ is a skew brace, where `$\circ$' is defined by  $a \circ b = a \cdot \lambda_a (b)$ for all $a, b \in G$. The following result is instrumental in our motive.

\begin{thm} \label{t1}
Let  $G$ be a group, $\lambda : G \to \mathrm{Aut}\, G$ be a homomorphism of
 $G$ into the group of its automorphisms. Then the set
$$ H_{\lambda} := \left\{ \, (\lambda_a, a) \, | \, a \in G  \, \right\}$$
is a subgroup of  $\mathrm{Hol}\, G= \mathrm{Aut}\, G \ltimes G$
if and only if
\begin{equation} \label{inc}
[G, \lambda(G)] := \{ b^{-1} \lambda_a (b) \mid a, b \in G \} \subseteq \mathrm{Ker}\,\lambda.
\end{equation}
Moreover, if  $H_{\lambda}$ is a subgroup, then it is  regular, and therefore by Theorem \ref{gv2017} we get a skew brace $(G, \cdot, \circ)$, where `$\circ$' is defined by $a \circ b = a \lambda_a(b)$.
\end{thm}

\begin{proof}
Suppose that  $H_{\lambda}$ is a subgroup of  $\mathrm{Hol}\, G$. Then, for all $a, b \in G $, the product
$$ (\lambda_a, a)(\lambda_b, b)=(\lambda_a \lambda_b, a \lambda_{a} (b)) =(\lambda_{ab}, a \lambda_a (b)) \in H_{\lambda},$$
which implies that
\begin{equation} \label{2}
\lambda(a \lambda_a(b))=\lambda_{a b}.
\end{equation}
Since $\lambda$ is a homomorphism, we get
\begin{equation} \label{3}
\lambda(\lambda_a(b))=\lambda_b,
\end{equation}
which is equivalent to
\begin{equation} \label{4}
\lambda(b^{-1} \lambda_a(b)) = id.
\end{equation}
Hence $[G, \lambda(G)]  \subseteq \mathrm{Ker}\,\lambda$.

Conversely, suppose that \eqref{inc} holds. Then  (\ref{4}), (\ref{3}) and (\ref{2}) hold, which imply that
$H_{\lambda}$ is closed under the operation of $\mathrm{Hol}\, G$. It remains  to check that  $H_{\lambda}$ is closed under inversion in $\mathrm{Hol}\, G$.
Notice that for   $(f, x)\in \mathrm{Hol}\, G$, we have  $(f, x)^{-1}=(f^{-1}, f^{-1}(x^{-1})).$
So
$$
(\lambda_a, a)^{-1}=\left( \lambda_a^{-1}, \lambda_a^{-1}(a^{-1})    \right).
$$
By \eqref{3}, we have   $\lambda\left( \lambda_{a^{-1}}(a^{-1})  \right)=\lambda_{a^{-1}}$.
Since $\lambda$ is a homomorphism, we finally get
$$
(\lambda_a, a)^{-1}=\left( \lambda_{a^{-1}}, \lambda_{a^{-1}}(a^{-1})  \right)=
\left( \lambda ( \lambda_{a^{-1}}(a^{-1}) ), \lambda_{a^{-1}}(a^{-1})  \right) \in H_{\lambda}.
$$
This proves that $H_{\lambda}$ is a subgroup of $\mathrm{Hol}\, G$.

Now we prove the final assertion. Let  $H_{\lambda}$ be a subgroup of $\mathrm{Hol}\, G$. By the definition of $H_{\lambda}$, the projection
$$
\pi_2 : H_{\lambda} \to G, \quad (\lambda_a, a) \mapsto a, \quad a \in G,
$$
is a bijection. Hence, we only need to show  that  the action of $H_{\lambda}$ on  $G$ is free and  transitive, i.e., for any  $b \in G$,
there exists a unique element $(\lambda_a, a) \in H_{\lambda}$ such that
\begin{equation} \label{5}
a \lambda_a(b)=1.
\end{equation}
If there exists an element $(\lambda_a, a) \in H_{\lambda}$ such that \eqref{5} holds, then, by   \eqref{2},  $\lambda_{ab}=id$. Thus $a=xb^{-1}$ for some $x \in \mathrm{Ker}\,\lambda$. Given $a = x b^{-1}$ for any $x \in \mathrm{Ker}\,\lambda$, $a \lambda_a(b)=1$ if and only if $x b^{-1}= \lambda_{b^{-1}}(b^{-1})$, that is, $a=\lambda_{b^{-1}}(b^{-1})$.
This establishes the existence as well as the uniqueness of the element  $(\lambda_a, a) \in H_{\lambda}$ for a given $b \in G$. Hence $H_{\lambda}$ is regular, and the proof is complete.
\end{proof}

\begin{cor} \label{c1}
If  $G\neq 1$ and $\mathrm{Ker}\,\lambda=1$, then
$H_{\lambda}$ is not a group.
\end{cor}

\begin{proof}
Indeed, the  hypothesis  $\mathrm{Ker}\,\lambda =1$, using \eqref{4}, gives
$$ \lambda_a(b)=b \; \mbox{ for all } a, b \in G.$$
 Hence  $\lambda_a = id$ for any  $a \in G$, which implies that   $\mathrm{Ker}\,\lambda=G$,  a contradiction.
\end{proof}

\medskip

We remark that the preceding theorem has already been proved in \cite{CC} for finite groups as a special case of so called `gamma functions'.  The following example shows that $H_{\lambda}$ can very well be a subgroup of $\mathrm{Hol}\, G$ for an anti-homomorphism $\lambda : G \to \Aut(G)$ with $\mathrm{Ker}\,\lambda =1$.

\begin{example}
Consider a subgroup  $D$ of $\mathrm{Hol}\,F_n$:
$$D= \left\{ \, (\iota_{a}, a) \, | \,  a \in   F_n   \, \right\}, $$
where $\iota_{a}$ is the inner automorphism of $F_n$ induced by $a$, that is,
  $$
\iota_{a}(b)=a^{-1}ba, \quad b \in F_n.
$$
Since
$$
(\iota_{a},a)(\iota_{b},b)=(\iota_{a}\iota_{b},a\iota_{a}(b))=
(\iota_{ba},ba),
$$
it follows that  $D$ is a subgroup of $\Hol F_n$. Notice that $\iota : F_n \to \Aut(F_n)$, $a \mapsto \iota_a$, is an anti-homomorphisms and $\Ker \iota = 1$.

Also the group operation  `$\circ$' on the set  $F_n$ defined by the formula
$$
a\circ b=a \; \iota_{a}(b)=ba
$$
gives a skew brace $(F_n, \cdot, \circ)$.
\end{example}

\begin{definition}
A skew brace $G$ is said to be {\it meta-trivial}  if there exists a trivial sub-skew brace $H$ of $G$ such that the skew brace $G/H$ is trivial too.
\end{definition}

We now present a characterization of $\lambda$-homomorphic skew braces.

\begin{thm}
Any $\lambda$-homomorphic skew-brace is meta-trivial.
\end{thm}
\begin{proof}
Let $(G, \cdot, \circ )$ be a $\lambda$-homomorphic skew brace. Notice that  $\lambda : (G, \circ) \to \Aut (G, \cdot)$ is  a  group homomorphism given by $\lambda_a(b) =  a^{-1} (a \circ b)$. By the given hypothesis, $\lambda : (G, \cdot) \to \Aut (G, \cdot)$ is also a group homomorphism.
 Define
$$G_0 = \left\{  a \in G \mid a \circ b = a b \, \mbox{  for all } b \in G  \, \right\}.$$
Notice that $G_0 = \Ker  \lambda$, and therefore it is a normal subgroup of both $(G, \circ)$ as well as $(G, \cdot)$. Obviously, $G_0$  is the trivial sub-skew brace of $G$.

Set $\bar G = (G, \cdot)/G_0$. Obviously, for all $a, b \in (G, \cdot)$,
$$ a G_0 \cdot bG_0= ab G_0.$$
Define
$$ aG_0 \circ bG_0 = ab G_0, $$
where $a,b \in (G, \cdot)$. Notice that $\{(\lambda_a, a) \mid a \in G\}$ is a regular subgroups of $\Hol (G, \cdot)$. Now by Theorem \ref{t1} we have
\[\lambda \big(b^{-1} \lambda_a(b)\big) = id.\]
Since
\begin{equation*} \label{3.10}
a \circ b=ab b^{-1} \lambda_a(b) \; \mbox{ for all } a, b \in (G, \circ),
\end{equation*}
it follows that the operation `$\circ$'  is well defined on $\bar G$. This makes $(\bar G, \cdot, \circ)$ a trivial skew brace, and the proof is complete.
\end{proof}

It is natural to ask
\begin{qns}
Is it true that every meta-trivial skew brace is  $\lambda$-homomorphic?
\end{qns}

In the following result we present a reduction argument, which allows us to verify  conditions on generators of a given group and on the images of the generators under $\lambda$.

\begin{proposition}\label{redprop}
Let $G$ be a group generated by  $X=\left\{ x_i \mid i\in I\right\}$ and $\lambda : G \to \Aut G$, $a \mapsto \lambda_a$,  a homomorphism such that $\lambda_{x_i} = \varphi_i$ for $i \in I$. If
\begin{equation*} \label{6}
x_i^{-1} \varphi_j(x_i) \in \Ker \lambda, ~~i, j \in I,
\end{equation*}
then
\begin{equation*} \label{7}
b^{-1} \lambda_a (b) \in \Ker \lambda, ~~ a, b \in G.
\end{equation*}
\end{proposition}
\begin{proof}
Let $G$  be generated by  $X=\left\{ x_i ~|~ i\in I\right\}$.   Assume that $x_i^{-1} \varphi_j(x_i) \in \Ker \lambda$ for all $i, j \in I$. Let $\alpha \in \Aut G$.
Since $x\alpha(x^{-1}) = x(\alpha(x))^{-1}=x (x^{-1}\alpha(x))^{-1} x^{-1}$ for all $x \in G$, it follows that 
\begin{equation}\label{eqn1sec2}
x^{-1}\alpha(x) \in \Ker \lambda \Longleftrightarrow x\alpha(x^{-1})  \in \Ker \lambda.
\end{equation}
 Next, since 
$x^{-1}\alpha^{-1}(x)= \left( \alpha^{-1}(x^{-1}) \alpha (\alpha^{-1}(x)) \right)^{-1}$ for all $x \in G$, we have 
\begin{equation}\label{eqn2sec2}
x^{-1}\alpha(x)  \in \Ker \lambda \Longleftrightarrow x^{-1}\alpha^{-1}(x)  \in \Ker \lambda.
\end{equation}

If $a$ and $b$ are arbitrary non-trivial elements of  $G$, then
$$
a=x_{i_{1}}^{\epsilon_1} \cdots  x_{i_{k}}^{\epsilon_k}, \quad
b=x_{j_{1}}^{\varepsilon_1} \cdots  x_{j_{m}}^{\varepsilon_m},
$$
where  $i_1,\ldots,i_k,j_1,\ldots,j_m \in I$, $\epsilon_1,\ldots,\epsilon_k,\varepsilon_1,\ldots,\varepsilon_m=\pm 1$.
Thus
$$
b^{-1}\lambda_a(b)=x_{j_{m}}^{-\varepsilon_m} \cdots  x_{j_{1}}^{-\varepsilon_1}
\varphi_{i_{1}}^{\epsilon_1} \cdots  \varphi_{i_{k}}^{\epsilon_k}
(x_{j_{1}}^{\varepsilon_1}\cdots  x_{j_{m}}^{\varepsilon_m}).
$$

Our proof goes by induction on  $n=k+m$. If $k+m=2$, then $k=1$, $m=1$, and therefore the assertion holds by the given hypothesis, \eqref{eqn1sec2} and \eqref{eqn2sec2}. Now we assume that $n = k+m>2$ and the result holds for all values  $\le n-1$. Then  either $k > 1$ or $m > 1$. Let $\gamma \in \Aut G$. Notice that $x^{-1} (\alpha\gamma)(x)=(x^{-1}\gamma(x)) (\gamma(x)^{-1} (\alpha \gamma)(x))$ for all $x \in G$. Thus, if both $x^{-1}\alpha(x)$ and $x^{-1}\gamma(x)$ lie in $\Ker \lambda$  for all $x \in G$, then $x^{-1} (\alpha \gamma)(x)  \in \Ker \lambda$. Also notice that $(xy)^{-1}\alpha(xy) = (y^{-1}(x^{-1}\alpha(x))y)(y^{-1}\alpha(y))$ for all $x, y \in G$. As a consequence,  if both $x^{-1}\alpha(x)$ and $y^{-1}\alpha(y)$ lie in  $\Ker \lambda$ for  any $x, y \in G$, then $(xy)^{-1}\alpha(xy)  \in \Ker \lambda$. Consequently, an easy computation shows that $b^{-1} \lambda_a (b) \in \Ker \lambda$ for  the number $n$ under consideration. The proof is now complete by induction. 
\end{proof}

\medskip

We now construct certain $\lambda$-homomorphic skew braces on free groups and free abelian groups.

Let $F_n = \gen{ x_1,\ldots,x_n }$  be the free group with free generators  $x_1, x_2, \ldots, x_n$.
Let  $\varphi_1, \ldots, \varphi_n \in \mathrm{Aut}\, F_n$ be a set of pairwise commuting automorphisms such that 
$$
x_i^{-1}\varphi_j(x_i)\in F_n',\quad i,j=1,\ldots,n.
$$
Define a homomorphism  $\lambda : F_n \to \Aut  F_n$ by the action on the generators:
$$
\lambda(x_i) :=  \lambda_{x_i} = \varphi_i, \quad i=1,\ldots,n.
$$
Then $\Ker \lambda$ contains the commutator subgroup  $F_n'$ of $F_n$.
Thus we have 
$$
x_i^{-1}\varphi_j(x_i)\in \Ker \lambda \quad i,j=1,\ldots,n
$$
and therefore, by Proposition \ref{redprop},  it follows that
$$
b^{-1} \lambda_a (b)\in \Ker \lambda
$$
for all  $a, b \in F_n$.
Hence, by Theorem \ref{t1},
$$
H_{\lambda}= \left\{ \, (\lambda_a, a) \, | \, a \in F_n  \, \right\}
$$
is a regular subgroup of  $\mathrm{Hol}\, F_n$. Thus the algebraic system  $B=( F_n, \cdot, \circ  )$,
where  $a \circ b=a \lambda_a(b)$ for $a, b\in F_n$, is a skew-brace.

Let $G$ be a group. Then an automorphism $\alpha$ of $G$ is said to be an IA-automorphism if $a^{-1}\alpha(a) \in G'$ for all $a \in G$. The set of all IA-automorphisms of $G$ constitutes a group, which we denote by $\mathrm{IA} \; G$.   So the above construction  is valid for  any choice of pairwise commuting automorphism $\varphi_1, \ldots, \varphi_n$  in $\mathrm{IA}\, F_n$. We now carry-out  the construction in a concrete case for $n=4$.

\begin{example}
Let $F_4 = \gen{x_1, x_2, x_3, x_4}$.
For $u, \; v \in \gamma_2(\gen{x_2, x_3})$, we can define automorphisms $\varphi_i$, $1 \le i \le 4$, of  $F_4$ just by their action of the free generators as follows:
$$
\varphi_1=\varphi_2 : \left\{
\begin{array}{ll}
x_1 \to x_1 u, &  \\
x_i \to x_i & \mbox{for} ~i \not= 1,
\end{array}
\right.
~~~
\varphi_3=\varphi_4 : \left\{
\begin{array}{ll}
x_i \to x_i & \mbox{for} ~i \not= 4,\\
x_4 \to x_4 v. &  \\
\end{array}
\right.
$$
It is not difficult to check that $\varphi_1 \varphi_3 = \varphi_3 \varphi_1$. Hence by the preceding discussion, we get a skew-brace $(F_4, \cdot \; , \circ)$, in which the operation `$\circ$' on free generators is given by
\begin{eqnarray*}
& x_1 \circ x_1=x_1^2u, \quad x_1 \circ x_2=x_1x_2,\quad x_1 \circ x_3=x_1x_3,\quad x_1 \circ x_4=x_1x_4,\\
& x_2 \circ x_1=x_2x_1u,\quad x_2 \circ x_2=x_2^2,\quad x_2 \circ x_3=x_2x_3,\quad x_2 \circ x_4=x_2x_4,\\
&x_3 \circ x_1=x_3x_1,\quad x_3 \circ x_2=x_3x_2,\quad x_3 \circ x_3=x_3^2,\quad x_3 \circ x_4=x_3x_4v,\\
&x_4 \circ x_1=x_4x_1,\quad x_4 \circ x_2=x_4x_2,\quad x_4 \circ x_3=x_4x_3,\quad x_4 \circ x_4=x_4^2v.
\end{eqnarray*}
The operations  `$\cdot$' and  `$\circ$' coincide  on  $F_4'$.
Further, if $\bar x$ is the inverse element under the operation `$\circ$', then $\bar x =\varphi^{-1}(x^{-1})$ and
inverses to $x_1,x_2,x_3,x_4$ are equal to:
$$
\bar x_1 = ux_1^{-1},\quad
\bar x_2 = x_2^{-1},\quad
\bar x_3 = x_3^{-1},\quad
\bar x_4 = v x_4^{-1}.
$$
Since
$$
F_4= \bigsqcup x_1^{a_1}x_2^{a_2}x_3^{a_3}x_4^{a_4} F_4',
$$
where  $a_1, a_2, a_3, a_4 \in \mathbb{Z}$, and
$$
x_1^{a_1}\circ x_2^{a_2}\circ x_3^{a_3}\circ x_4^{a_4}\equiv
x_1^{a_1}x_2^{a_2}x_3^{a_3}x_4^{a_4}\, (\, \mathrm{mod} \, F_4'),
$$
then it follows that
$$
(F_4, \circ)= \bigsqcup x_1^{a_1}\circ x_2^{a_2}\circ x_3^{a_3}\circ x_4^{a_4} (F_4, \circ)'
$$
is a decomposition of  $(F_4, \circ)$ into disjoint union of cosets of $(F_4, \circ)'$. Hence,
$$
(F_4, \circ)/(F_4, \circ)' \cong \mathbb{Z}^4.
$$
\end{example}

\bigskip

\begin{qns}
Find a presentation of  $(F_4, \circ)$. Is it true that $F_4 \cong (F_4, \circ)$?
\end{qns}
The following conjugation relations in  $\left\langle F_4, \circ  \right\rangle$ may be useful while attempting this question: 
\begin{eqnarray*}
& ~~ \bar x_1 \circ x_2 \circ x_1 = ux_1^{-1}x_2x_1, \quad \bar x_1 \circ x_3 \circ x_1= ux_1^{-1}x_3x_1u^{-1}, \quad \bar x_1 \circ x_4 \circ x_1= ux_1^{-1}x_4x_1u^{-1},\\
& \bar x_2 \circ x_1 \circ x_2 = x_2^{-1}x_1u^{-1}x_2, ~~~~\quad \bar x_2 \circ x_3 \circ x_2= x_2^{-1}x_3x_2,~~~~ \quad \bar x_2 \circ x_4 \circ x_2= x_2^{-1}x_4x_2,~~~~ ~~\quad \quad\\
& \bar x_3 \circ x_1 \circ x_3 = x_3^{-1}x_1x_3, ~~~~\quad \bar x_3 \circ x_2 \circ x_3= x_3^{-1}x_2x_3, ~~~~\quad \bar x_3 \circ x_4 \circ x_3= x_3^{-1}x_4v^{-1}x_3,~~~~  ~~\quad  \quad\\
& \bar x_4 \circ x_1 \circ x_4  = vx_4^{-1}x_1x_4v^{-1}, \quad \bar x_4 \circ x_2 \circ x_4= vx_4^{-1}x_2x_4v^{-1}, \quad \bar x_4 \circ x_3 \circ x_4= vx_4^{-1}x_1x_4.
\end{eqnarray*}

\bigskip

Now we construct  braces on free abelian groups.
\begin{example}
Let $A  \cong \mathbb{Z}^n = \langle x_1, x_1, \ldots, x_n \rangle$ be a free abelian group of rank $n$.
Define $\varphi_i \in \Aut A$ as follows:
$$
\varphi_i :
\left\{
\begin{array}{ll}
x_i \to x_i + x_n, &  \\
x_j \to x_j & ~\mbox{for} ~j \not= i,
\end{array}
\right.
$$
for  $1 \le i \le n-1$ and
$$\varphi_n = id.$$
Notice that these automorphisms  commute pairwise, and
$$
\varphi_i (x_i) - x_i = x_n, \quad \varphi_i (x_j) - x_j = 0, \quad \mbox{if}~j \not=i.
$$
Hence, the homomorphism $\lambda  : A \to \mathrm{Aut}(A)$,  defined on the free generators $x_1,x_2, \ldots, x_n$ by
$$
\lambda(x_i)=\varphi_i, \quad i=1, 2, \ldots, n,
$$
satisfies
$$
\lambda(\varphi_i(x_j)-x_j)=id, \quad i,j=1, 2, \ldots, n.
$$
Thus, by Proposition \ref{redprop} and  Theorem \ref{t1}, 
$$
H_{\lambda }= \left\{ \, (\lambda_a, a) \, | \, a \in A  \, \right\}=
 \left\{ \, \left(\prod_{i=1}^{n-1} \varphi_i^{\alpha_i}, \sum_{i=1}^n \alpha_i x_i \right) \mid \alpha_i \in \mathbb{Z}  \, \right\}
$$
is a regular subgroup of  $\mathrm{Hol}\, A$, where $a = \sum_{i=1}^n \alpha_i x_i$, and therefore the operation
$$
a \circ b = a+ \lambda_a(b),\quad a, b \in A
$$
gives rise to a  brace structure  $(A, + , \circ )$.

We now investigate  the structure of the group  $(A, \circ )$.
Notice that the operations  `$+$' and  `$\circ$' are equal on the set
$$
A_0 := \mathrm{Ker}\, \lambda =\left\langle x_n\right\rangle.
$$
Further, for any
$$
a =  \sum_{i=1}^n \alpha_i x_i, ~~~b = \sum_{i=1}^n \beta_i x_i \in A
$$
the following holds
\begin{eqnarray*}
a \circ b &=& a + \lambda_a(b) \\
&=& (\alpha_1 + \beta_1) x_1 + (\alpha_2 + \beta_2) x_2 + \cdots +(\alpha_{n-1} + \beta_{n-1}) x_{n-1} +\\
& & (\alpha_n + \beta_n + \alpha_1 \beta_1 + \alpha_2 \beta_2 + \cdots + \alpha_{n-1} \beta_{n-1}) x_n\\
& =&  b \circ a.
\end{eqnarray*}
Hence, $(A, \circ)$ is an abelian group in which the powers of the generators are given by
$$
x_i^{\circ k} = k x_i + \frac{k(k-1)}{2}x_n,\;1 \le i \le n-1
$$
and
$$
x_n^{\circ k} = k x_n,
$$
for all $k \in \mathbb{Z}$.
Thus any element  $b = \sum_{i=1}^n \beta_i x_i \in A$ has the unique form
$$
b = x_1^{\circ \alpha_1} \circ x_2^{\circ \alpha_2} \circ \cdots \circ x_n^{\circ \alpha_n},
$$
where $\alpha_i = \beta_i$ for $ 1 \le  i \le n-1$ and
$$
\alpha_n = \beta_n - \sum_{j=1}^{n-1} \frac{\beta_j (\beta_j - 1 )}{2}.
$$
Hence  $(A, \circ) \cong A$ and  we have the following short exact sequence
$$
1 \to \Ker \lambda \to A^{(\circ)} \to Z^{n-1} \to 1,
$$
where $\mathrm{Ker}\, \lambda = \langle x_n \rangle \cong \mathbb{Z}$.
\end{example}





\bigskip


\section{$\lambda$-cyclic skew braces} \label{lambda-cyc}

In this section we study $\lambda$-cyclic skew braces, and construct such skew braces on various homogeneous groups.  We begin by recalling
\begin{definition}
A $\lambda$-homomorphic skew brace $(G, \cdot, \circ)$ is said to be a $\lambda$-{\it cyclic skew brace} if the image $\im \lambda$ is a cyclic subgroup of $\Aut (G, \cdot)$.
\end{definition}

Let $G = (G, \cdot)$ be a group defined by a presentation
$$
\mathcal{P} = \langle \mathcal{A} \mid\mathcal{R} \rangle,
$$
where $\mathcal{A} = \{ g_1, g_2, \ldots \}$ is the set of generators and $\mathcal{R} = \{ r_1, r_2, \ldots \}$ is the set of relations. Then any element of $w$ of $G$ can be presented as a word
\begin{equation} \label{w}
w = g_{i_1}^{\alpha_1} g_{i_2}^{\alpha_2} \ldots g_{i_s}^{\alpha_s},~~~g_{i_j} \in \mathcal{A},~~\alpha_j \in \mathbb{Z},
\end{equation}
in  $\mathcal{A}^{\pm 1} = \{ g_1^{\pm 1}, g_2^{\pm 1}, \ldots \}$. We define the {\it logarithm} of $w$, denote by $l(w)$, as  the sum
$$
l(w) = \sum_{j=1}^s \alpha_j.
$$

We say that the presentation $\mathcal{P}$ is {\it homogeneous} if  $l(r_i) = 0$ for all $r_i \in \mathcal{R}$. We  say that a group is {\it homogeneous} if it can be present by a homogeneous presentation. It is not difficult to see that for a homogeneous group $G$, the function logarithm $l : G \to \mathbb{Z}$ is well defined. Also notice that a homogeneous group can have a torsion. For example, the relation $(g_i g_j^{-1})^k = 1$, $k > 1$, has zero logarithm,  but the element $g_i g_j^{-1}$ has finite order.

\begin{example}
1) Any free group, free solvable or free nilpotent group is homogeneous.

2) If $L$ is a tame link in the 3-sphere  $\mathbb{S}^3$, then Wirtinger presentation of the group $G_L = \pi(\mathbb{S}^3 \setminus L)$ gives a homogeneous presentation.
\end{example}

Suppose that $(G, \cdot)$ is a homogeneous group, which is presented by a homogeneous presentation $\mathcal{P}$. Let $\varphi \in \mathrm{Aut} (G, \cdot)$ be such that
$$
l(\varphi (g_i)) = 1~\mbox{for all}~g_i \in \mathcal{A}.
$$

With this setup we have the following result.

\begin{lemma} \label{l1}
1) $l(u v) = l(u) + l(v)$;

2) $l(u^{-1}) = -l(u)$;

3) $l (\varphi^k (w)) = l(w)$ for any integer $k$.
\end{lemma}
\begin{proof}
The first and second formulas follow from the definition of the logarithm. Obviously, it is sufficient to prove the  third formula for  $k=1$. Suppose that $w$ has the form (\ref{w}), then, using the first two formulas and the property of  $\varphi$, we get
$$
\varphi(w) = \varphi(g_{i_1})^{\alpha_1} \varphi(g_{i_2})^{\alpha_2} \cdots \varphi(g_{i_s})^{\alpha_s}
$$
and
\begin{eqnarray*}
l(\varphi(w)) &=& \alpha_1 l(\varphi(g_{i_1})) + \alpha_2 l(\varphi(g_{i_2})) + \cdots + \alpha_s l(\varphi(g_{s_1})) \\
&=& \alpha_1  + \alpha_2  + \cdots + \alpha_s\\
& =& l(w).
\end{eqnarray*}
The proof is complete.
\end{proof}

\vspace{.1in}

We continue with the above setup. Notice that the map
$$
\lambda : (G, \cdot) \to \mathrm{Aut} (G, \cdot),
$$
given  by  $w \mapsto \lambda_w := \varphi^{l(w)}$, defines a homomorphism.
By Lemma \ref{l1}(3) we easily get   $\lambda(w^{-1} \lambda_a(w)) = id$.  It now follows from  Theorem \ref{t1} that   $H_{\lambda} = \{ (\varphi^{l(w)}, w) \in \mathrm{Hol}~ (G, \cdot) \mid w \in G \}$ is a regular subgroup of
$\Hol (G, \cdot)$.  A direct proof of this statement is also easy, which we present below.

\begin{thm} \label{lcb}
The set
$
H_{\lambda} = \{ (\varphi^{l(w)}, w) \in \mathrm{Hol}~ (G, \cdot) \mid w \in G \}
$
is a regular subgroup of $\mathrm{Hol}~ (G, \cdot)$, where $(G, \cdot)$ and $\varphi$ are as taken above.
\end{thm}
\begin{proof}
Let $(\varphi^{l(u)}, u)$ and  $(\varphi^{l(v)}, v)$ be two elements in $H_{\varphi}$. Then
$$
(\varphi^{l(u)}, u) (\varphi^{l(v)}, v) = (\varphi^{l(u) + l(v)}, u \varphi^{l(u)} (v)).
$$
 By Lemma \ref{l1}  we have
$$
l(u \varphi^{l(u)}(v)) = l(u) + l(\varphi^{l(u)}(v)) = l(u) + l(v).
$$
Hence,
$$
(\varphi^{l(u)}, u) (\varphi^{l(v)}, v) \in H_{\varphi}.
$$

Let us find the inverse  $(\varphi^{l(u)}, u)^{-1}$ of  $(\varphi^{l(u)}, u)$ in $\Hol (G, \cdot)$. Suppose that $(\varphi^{l(u)}, u)^{-1} = (\varphi^{l(v)}, v)$ for some $v \in G$. Then
$$
(\varphi^{l(u) + l(v)}, u \varphi^{l(u)} (v)) =(\varphi^{l(u)}, u) (\varphi^{l(v)}, v) = (id, 1),
$$
where $id$ is the identity automorphism and $1$ is the identity element in $G$. This equality is equivalent to the system of equations
$$
\left\{
\begin{array}{l}
\varphi^{l(u) + l(v)} = id,\\
u \varphi^{l(u)} (v) = 1.\\
\end{array}
\right.
$$
From  the second equality, we get
$$
v = \varphi^{-l(u)} (u^{-1}).
$$
Lemma \ref{l1} now gives
$$
l(v) = l\left( \varphi^{-l(u)} (u^{-1}) \right) = l (u^{-1}) = - l (u),
$$
which satisfies the first equality.  Hence,  $(\varphi^{l(u)}, u)^{-1} \in H_{\varphi}$, and therefore $ H_{\varphi}$ is a subgroup of $\Hol \; (G, \cdot)$.

We now prove that $H_{\varphi}$ is  regular. It is required to show that for any given $u \in G$, there is a unique element $(\varphi^{l(v)}, v) \in H_{\varphi}$ such that
$$
v \varphi^{l(v)} (u) = 1.
$$
An easy computation shows  that $(\varphi^{-l(u)}, \varphi^{-l(u)} (u^{-1}))$ is the required unique element.
\end{proof}

We now present an example satisfying the hypothesis of the  preceding theorem.

\begin{example}
As observed above the free group $F_n$, on free generators $x_1, \ldots, x_n$, is homogeneous. Let $\theta$ be an $IA$-automorphism of  $F_n$. Then it follows that $l(\theta(x_i)) = 1$.  Hence, by the preceding theorem, the set
 $H_{\lambda}$, corresponding to the  homomorphism  $\lambda : F_n \to \Aut F_n$
 defined on the free generators by
 $$
\lambda(x_i)=\theta, \quad i=1,\ldots,n,
$$
is a regular subgroup of  $\Hol F_n$.
In particular,   any  arbitrary inner automorphism of  $F_n$ may be taken as $\theta$.
\end{example}

We remark that the preceding construction does not work for all automorphisms of $F_n$. We consider the following example.

\begin{example}
Let   $\theta$ be the automorphism of $F_n$ defined on the free generators by
$$
\theta(x_1)=x_1x_2,\,\, \theta(x_k)=x_k,\,\, k=2,\ldots,n.
$$
Then,
$$
\lambda(x_1^{-1}(\lambda(x_2)(x_1))) = \lambda(x_1^{-1} \theta(x_1)) = \lambda(x_2) = \theta,
$$
which is not an identity automorphism of $F_n$. Hence, by Theorem \ref{t1}, $H_{\lambda}$ is not a subgroup of $\Hol F_n$.
\end{example}

It is easy to prove the following result.

\begin{proposition}\label{impprop}
Let $\theta$  be  automorphisms of $F_n$ defined on the free generators $x_1, \ldots, x_n$ by
\[\theta(x_i) =  x_{\sigma(i)} u_i\,  \mbox{for some } u_i \in F_n'\; \mbox{and some permutation } \sigma \in S_n\]
or
\[\theta(x_i) =  x_{\sigma(i)}^{-1} u_i\,  \mbox{for some } u_i \in F_n'\; \mbox{and some permutation } \sigma \in S_n.\]
Then
$$
\lambda (b^{-1}\lambda_a(b))=1
$$
for all $a, b \in F_n$, where  $\lambda : F_n\to \Aut F_n$,
defined by the rule
$$
\lambda(x_i)=\theta, \quad i=1,\ldots,n,
$$
is a homomorphism.
\end{proposition}

In the remaining of this section, we concentrate on $F_2$. Let  $F_2$ be the free group on free generators $x, y$. Let $ \theta\in \Aut F_2$ which permutes the generators, that is, 
$$ \theta: x\leftrightarrows y.$$
Then $\theta$ satisfies the hypothesis of Proposition \ref{impprop}, and therefore $H_{\lambda}$ is a regular subgroup of  $\Hol F_2$.  By Reidemeister-Shraier method, it follows that  the subgroup $H_0 = \gen{xy, yx, x^2, y^2}$  is  the kernel of $\lambda$. Notice that  $H_0$ is the free group with free generators $xy$, $x^2$, $xy^{-1}$ and has index  2 in $F_2$. Since $F_2$ is a normal subgroup of $\Hol F_2$, it follows that $H_{\lambda}$ and $F_2$ are not conjugate in $\Hol F_2$.

Notice that $\lambda_a = id$ if $a \in H_0$ and $\lambda_a = \theta$ otherwise. Hence $( F_2, \cdot, \circ )$ is a skew brace with
$$
a\circ b=
\left\{
\begin{array}{ll}
  ab, & \mbox{if} \,\, a\in H_0,\,\, b \in F_2,\\
  \\
  a\theta(b), & \mbox{if} \,\, a\not\in H_0,\,\, b \in F_2. \\
\end{array}
\right.
$$

Since `$\circ$' and `$\cdot$' coincide on $H_0$, $(H_0, \circ)$ is a free group of rank $3$. With the preceding setting, we now prove

\begin{proposition}
$(F_2,  \circ)$ is a free group of rank 2.
\end{proposition}

\begin{proof}
It is enough to prove that  $(F_2,  \circ)$ does not have any torsion element.
Indeed, it follows from a well known
theorem (\cite{St1968}, \cite{Sw1969}) that a finitely generated group without
torsions admitting a free subgroup of finite index is free.
Further, it is enough to show that $(F_2,  \circ)$ does not contain any 2 torsion.
Indeed, the subgroup $H_0$, defined above,  is a normal free subgroup of  $(F_2,  \circ)$ of index two.
Thus, if $(F_2,  \circ)$ has an element of finite order, then $(F_2,  \circ)$ has an element of order two.
Contrarily assume that  $a \in F_2 \setminus \left\{ 1 \right\}$ such that $a\circ a= 1$.
Then $a=xa_0$, $a_0 \in H_0$  and $xa_0y \theta(a_0)= a \circ a = 1$.

Recall that  $l(w)$ denotes the sum of all powers of the generators $x,y$ in the word  $w\in F_2$.
Thus, $l(xa_0y \theta(a_0))=l(1)$ gives 
$$
2+ l(a_0) + l(\theta(a_0))=0.
$$
Note that   $\mathrm{log}(a_0)+\mathrm{log}(\theta(a_0))\equiv 0 (\mbox{mod}\, 4)$. Hence we get
$2\equiv 0 (\mbox{mod}\, 4)$, which is absurd. Hence $(F_2,  \circ)$ does not admit any torsion element.
Using that $\mathrm{rank}(H_0)=3$, $|(F_2,  \circ): H_0|=2$ and the well known formula
$$
\mathrm{rank}(H_0)=1+(\mathrm{rank}((F_2,  \circ))-1) \cdot |(F_2,  \circ): H_0|,
$$
we obtain  $\mathrm{rank}(F_2,  \circ)=2$. The proof is now complete.
\end{proof}

\medskip

Hence, we have proved the following statement.

\begin{proposition}
Let $F_2 = \gen{x, y}$ be a free group, $\theta$  an automorphism of  $F_2$ such that $\theta : x\leftrightarrows y$ and $H_0 \leq F_2$ is the kernel of the homomorphism  $F_2 \to \mathbb{Z}_2 = \gen{\theta}$,
$x,y \mapsto \theta$. Then $(F_2, \cdot, \circ )$ is a skew brace, where `$\circ$'  is defined by the rule
$$
a\circ b=
\left\{
\begin{array}{ll}
  ab, & \mbox{if} \,\, a\in H_0,\,\, b \in F_2,\\
  \\
  a\theta(b), & \mbox{if} \,\, a\not\in H_0,\,\, b \in F_2 \\
\end{array}
\right.
$$
and $(F_2, \circ)$ is a free group of rank 2.
\end{proposition}

\bigskip

We conclude this section by constructing a skew brace on $F_2 = \gen{x, y}$ by the automorphism $\theta$ defined on the generators by $x \mapsto x^{-1},\,y \mapsto y^{-1}$.
Let $\lambda :  F_2 \to \mathbb{Z}_2=\left\langle \theta \right\rangle$ be a homomorphism
such that $\lambda(x)= \theta = \lambda(y)$. Then, by Proposition \ref{impprop}, $H_{\lambda}$ is a regular subgroup of  $\Hol  F_2$. Notice  that  $H_0$, the subgroup of $F_2$ consisting of words of even length is equal to $\Ker \lambda$. Since $F_2$ is a normal subgroup of $\Hol F_2$, it follows that  $H$ and $F_2$ are not conjugate in $\Hol F_2$.

Notice that $\lambda_a = id$ if $a \in H_0$ and $\lambda_a = \theta$ otherwise. Hence $( F_2, \cdot, \circ )$ is a skew brace with
$$
a\circ b=
\left\{
\begin{array}{ll}
  ab, & \mbox{if} \,\, a\in H_0,\,\, b \in F_2,\\
  \\
  a\theta(b), & \mbox{if} \,\, a\not\in H_0,\,\, b \in F_2. \\
\end{array}
\right.
$$
Notice that $x \circ x = 1$.

We now find  the structure of  $(F_2, \circ)$. Notice that $(F_2, \circ)$ is generated by  $H_0$ and  $x$.
Since $x \circ x=1$ and $H_0 \unlhd F_2$, it follows that
$$
(F_2, \circ) = H_0 \rtimes \mathbb{Z}_2, \quad \mathbb{Z}_2=\left\langle s \right\rangle,
$$
where $s \in (F_2, \circ) \setminus H_0$. Notice that  $H_0$ is freely generated by
$$
p= xy,\quad q=x^2,\quad r=y^2.
$$
Also $\bar x = x$. To avoid confusion, let us denote $x$ in $(F_2, \circ)$ by $s$.
We get the following conjugation relations:
$$
sps=p^{-1},\quad sqs=q^{-1},\quad srs=pr^{-1}p^{-1}.
$$
Hence,  $(F_2, \circ)$ has the following  presentation
\begin{eqnarray*}
(F_2, \circ) &=&\left\langle\, p,q,r,s \mid s^2=1,\,\, sps=p^{-1},\,\, sqs=q^{-1},\,\, srs=pr^{-1}p^{-1}\, \right\rangle\\
&\cong& F_3 \rtimes \mathbb{Z}_2.
\end{eqnarray*}

We have proved

\begin{proposition}
Let  $F_2 = \gen{x, y}$ be the free group of rank $2$,   $\theta$  an automorphism of $F_2$ such that $\theta: x\mapsto x^{-1},\,y \mapsto y^{-1}$ and
$H_0\leq F_2$  the kernel of the homomorphism  $F_2 \to \mathbb{Z}_2=\left\langle \theta \right\rangle$, given by
$x, y \mapsto \theta$.
Then $(F_2, \cdot, \circ)$ is a  skew brace with
$$
a\circ b=
\left\{
\begin{array}{ll}
  ab, & \mbox{\hspace{.2in}} \,\, a\in H_0,\,\, b \in F_2,\\
  \\
  a\theta(b), & \mbox{\hspace{.2in}} \,\, a\not\in H_0,\,\, b \in F_2 \\
\end{array}
\right.
$$
Moreover,  $\left\langle F_2, \circ \right\rangle$ is presented by
$$
F_3 \rtimes \mathbb{Z}_2  = \left\langle\, p,q,r,s \mid s^2=1,\,\, sps=p^{-1},\,\, sqs=q^{-1},\,\, srs=pr^{-1}p^{-1}\, \right\rangle.
$$
\end{proposition}

\bigskip

\section{$\lambda$-cyclic braces on free abelian groups} \label{cyc-ab}

In this section we   construct all possible $\lambda$-cyclic braces  on the  free abelian group $\mathbb{Z}^2$ generated by $x_1$ and $x_2$. It turns out that there are   only $3$ non-isomorphic $\lambda$-cyclic braces. We also construct a non-trivial brace on  any free abelian group  of finite rank.

We start with a general case. Let us take the free abelian group of rank  $n$
$$
\mathbb{Z}^n=\left\{ \, m_1x_1+\cdots + m_nx_n \, | \, m_1, \ldots ,m_n \in   \mathbb{Z}   \, \right\}
$$
with free basis  $\{x_1,\ldots, x_n\}$. Recall that for $a = m_1x_1+\cdots + m_nx_n \in \mathbb{Z}^n$, we defined $l(a) = \sum_{i=1}^{n}m_i$. Let $\varphi$ be a fixed automorphism of $\mathbb{Z}^n$.   As usual, $H_{\lambda} = \{(\lambda_a, a) \mid a \in \mathbb{Z}^n\}$, where $\lambda_a= \varphi^{l(a)}$, is a subset of $\Hol \mathbb{Z}^n$. Let
\[\mathbb{Z}_0^n := \{a \in \mathbb{Z}^n \mid l(a) = 0\}.\]
 Then the next result  follows from Theorem \ref{lcb}.

\begin{proposition}
Suppose that  $\varphi \in \Aut \mathbb{Z}^n$ is such that $\varphi(x_i) \equiv x_i \; (\mathrm{mod} \, \mathbb{Z}^n_0)$, $i=1,\ldots,n$.
Then  $H_{\lambda}$ is a regular subgroup of  $\Hol \mathbb{Z}^n$.
\end{proposition}

We now prove the converse of this result. More generally

\begin{proposition}\label{cycprop2}
For the free abelian group $\mathbb{Z}^n$ with free basis $\{x_1, \ldots, x_n\}$, $H_{\lambda}$ is a regular subgroup of  $\Hol \mathbb{Z}^n$ if and only if $\varphi \in \Aut \mathbb{Z}^n$ is of the form $\varphi(x_i) \equiv x_i \; (\mathrm{mod} \, \mathbb{Z}^n_0)$, $i=1,\ldots,n$.
\end{proposition}
\begin{proof}
Let $\varphi \in \Aut \mathbb{Z}^n$ and  $H_{\lambda}$ be a regular subgroup of  $\Hol (\mathbb{Z}^n)$, where $\lambda_{x_i} = \varphi$ for $1 \le i \le n$. Then
$$
(\varphi, x_i) (\varphi, x_j)=(\varphi^2, x_i+\varphi(x_j))\in H_{\lambda}, \quad i,j=1,\ldots,n,
$$
which implies that $l(x_i+\varphi(x_i)) = 2$, for $1 \le i \le n$.
Hence, if
$$
\varphi(x_i)=m_{i1}x_1 + \cdots + m_{in}x_n,
$$
for some $m_{i1},\ldots , m_{in}\in \mathbb{Z}$, then  we get $m_{i1} + \cdots + m_{in}=1$ for $1 \le i \le n$.  Adding  and subtracting $x_i$ on the left side, we get
$$
\varphi(x_i)\equiv x_i \;(\mathrm{mod} \, \mathbb{Z}^n_0), \quad i=1,\ldots,n.
$$

Converse follows from the preceding proposition.
\end{proof}

As an application of the preceding  proposition, we get a left brace $(\mathbb{Z}^n, +, \circ)$, where the  operation  `$\circ$' is defined by
$$
a\circ b=a+\varphi^{l(a)}(b),\quad a,b \in \mathbb{Z}^n.
$$

\bigskip

We are now going to  construct all possible $\lambda$-cyclic braces  on the  free abelian group $\mathbb{Z}^2$. Notice that  $\mathbb{Z}^2_0$ is generated by the element $x_1-x_2$. If $\varphi\in \Aut \mathbb{Z}^2$ is such that
$$
\varphi(x_1) \equiv x_1 \;(\mathrm{mod} \, \mathbb{Z}^2_0), \quad
\varphi(x_2) \equiv x_2 \;(\mathrm{mod} \, \mathbb{Z}^2_0),
$$
then
$$
\varphi(x_1)=x_1+p(x_1-x_2),\quad \varphi(x_2)=x_2+q(x_1-x_2)
$$
for some integers  $p$, $q$ satisfying the relation
$$
\det [\varphi]=p-q+1=\pm 1,
$$
where $[\varphi]$ is the matrix of $\varphi$ with respect to  the basis $\{x_1, x_2\}$ of $\mathbb{Z}^2$. Hence, $q=p$ or $q=p+2$, and
$$
[\varphi]=\left(%
\begin{array}{cc}
  1+p & -p \\
  p & 1-p
\end{array}%
\right)
\quad
\mbox{or}
\quad
[\varphi]=\left(%
\begin{array}{cc}
  1+p & -p \\
  p+2 & -1-p
\end{array}%
\right).
$$
Notice that in the first case  $\varphi$ has infinite order and
$$
[\varphi]=\left(%
\begin{array}{cc}
  1+p & -p \\
  p & 1-p
\end{array}%
\right)
=\left(%
\begin{array}{cc}
  2 & -1 \\
  1 & 0 \\
\end{array}%
\right)^p.
$$
And in the second case, the order of  $\varphi$ is $2$, that is,
$$
\left(%
\begin{array}{cc}
  1+p & -p \\
  p+2 & -1-p \\
\end{array}%
\right)^2=E.
$$

\medskip

Let us first  consider the case
$$
[\varphi]=\left(%
\begin{array}{cc}
  1+p & -p \\
  p & 1-p \\
\end{array}%
\right).
$$
In this case  the automorphism $\varphi$ takes the following form:
$$
\varphi = \left\{
\begin{array}{l}
x_1 \mapsto (1+p) x_1 - p x_2,\\
x_2 \mapsto p x_1 + (1-p) x_2.\\
\end{array}
\right.
$$

Using induction on $k$, it is easy to prove

\begin{lemma}
For any integer $k$, the following holds
$$
\varphi^k = \left\{
\begin{array}{l}
x_1 \mapsto (1 + k p) x_1 - k p x_2,\\
x_2 \mapsto k p x_1 + (1 - k p) x_2.\\
\end{array}
\right.
$$
Consequently,
$$
[\varphi]^k = \left(%
\begin{array}{cc}
  1 + k p & -k p \\
  k p & 1 - k p \\
\end{array}%
\right).
$$
\end{lemma}

\medskip

Take elements  $a=a_1x_1+a_2x_2$ and  $b=b_1x_1+b_2x_2$ in
$\mathbb{Z}^2$ and define the operation
$$
a\circ b=a+\varphi^{a_1+a_2}(b) = a + b + p (a_1 + a_2)(b_1 + b_2)  (x_1 - x_2).
$$
We see that if  $a_1+a_2=0$, then $a\circ b=a+b$, that is, `$\circ$' and `$+$' coincide
on $\mathbb{Z}^2_0$. As we already know that  $(\mathbb{Z}^2, \circ)$ is a group. It is easy to see that  ${\bf 0} = 0x_1 + 0x_2$ is the identity element of $(\mathbb{Z}^2, \circ)$.  Suppose that $a \circ b = {\bf 0}$, then we have the following system of equations:
\begin{eqnarray*}
& a_1 + b_1 +  p (a_1+a_2) (b_1 + b_2) = 0,\\
& a_2 + b_2 -  p (a_1+a_2) (b_1 + b_2) = 0.
\end{eqnarray*}
This implies that a right inverse $a_r$ of $a$  has the form
$$
a_r = -a + p (a_1+a_2)^2 (x_1 - x_2).
$$
An easy calculation shows that $a_r$ is also the left inverse of  $a$. Thus we have established
\begin{lemma}
The inverse of the element $a=a_1x_1+a_2x_2 \in (\mathbb{Z}^2, \circ)$  has the form:
$$
\bar a =
\left\{\begin{array}{ll}
  -a, & \mbox{if}\,\,\, a_1 + a_2 = 0,  \\
  \\
 -a + p (a_1+a_2)^2 (x_1-x_2), & \mbox{if}\,\,\, a_1 + a_2 \not= 0. \\
\end{array}
\right.
$$
\end{lemma}

\medskip

We are now going to investigate the structure of $(\mathbb{Z}^2, \circ)$. Notice that
$$
(a_1 x_1 + a_2 x_2) \circ (m x_1) = (a_1 + m + m p(a_1 + a_2)) x_1 + (a_2 - m p (a_1 + a_2)) x_2.
$$
If $a_1 + a_2 = 0$, then
$$
(a_1 x_1 + a_2 x_2) \circ (m x_1) = (a_1 + m) x_1 + a_2 x_2.
$$
Hence
$$
\mathbb{Z}^2=\bigsqcup\limits_{m\in \mathbb{Z}}  \mathbb{Z}^2_0 \circ (m x_1)
$$
is a decomposition of $\mathbb{Z}^2$ as the disjoint union of cosets of $\mathbb{Z}^2_0$.

Further, $m$-th power of  $x_1$ in $(\mathbb{Z}^2, \circ)$ has a form
$$
x_1^{\circ m}=x_1\circ \cdots \circ x_1 = ( id + \varphi + \cdots + \varphi^{m-1})(x_1).
$$
Since
$$
E + [\varphi]+ \cdots + [\varphi]^{m-1}=
\left(%
\begin{array}{cc}
  m+\frac{m(m-1)}{2} p & -\frac{m(m-1)}{2} p \\
  \frac{m(m-1)}{2} p & m-\frac{m(m-1)}{2}p \\
\end{array}%
\right),
$$
we get
$$
x_1^{\circ m}= \left(m + \frac{m(m-1)}{2} p \right) x_1  - \frac{m(m-1)}{2} p x_2,
$$
which gives
$$
x_1^{\circ m}-mx_1 = \frac{m(m-1)}{2} p  x_1  - \frac{m(m-1)}{2} p x_2 = \frac{m(m-1)}{2} p (x_1 - x_1) \in \mathbb{Z}^2_0.
$$

We also have
$$
x_1 \circ (x_1 - x_2) = 2 x_1 - x_2
$$
and
$$
(x_1 - x_2) \circ  x_1  = 2 x_1 - x_2.
$$
Thus, it follows that
$$
(\mathbb{Z}^2, \circ) = \bigsqcup\limits_{m\in \mathbb{Z}}  \mathbb{Z}^2_0 \circ (m x_1)
$$
is the decomposition of $(\mathbb{Z}^2, \circ)$ into the disjoint union of cosets of $\mathbb{Z}^2_0$. In particular,
$(\mathbb{Z}^2, \circ)$ is generated by elements $x_1$ and $x_1 - x_2$. Since
$$
x_1 \circ (x_1 - x_2) = (x_1 - x_2) \circ  x_1,
$$
it follows that $(\mathbb{Z}^2, \circ)$ is the free abelian of rank 2.

\bigskip

Now we consider the second  case, that is,
$$
[\varphi] =\left(%
\begin{array}{cc}
  1 + p & -p \\
  2 + p & -1 - p \\
\end{array}%
\right),~~~p \in \mathbb{Z}.
$$
This matrix defines the automorphism $\varphi$ of the free abelian group $\mathbb{Z}^2 = \langle x_1, x_2 \rangle$ given by
$$
\varphi = \left\{
\begin{array}{l}
x_1 \mapsto (1+p) x_1 -p x_2,\\
x_2 \mapsto (2+p) x_1 - (1+p) x_2.\\
\end{array}
\right.
$$

Let  $a=a_1x_1+a_2x_2$, $b=b_1x_1+b_2x_2$ be two elements of
$\mathbb{Z}^2$, then using the formula $a\circ b = a + \varphi^{a_1+a_2}(b)$, we get
$$
a\circ b  =
\left\{\begin{array}{ll}
  a+b, & \mbox{if}\,\,\, a_1+a_2 \equiv 0 \pmod{2},  \\
  \\
  a + b + (2 b_2 + p (b_1 + b_2)) (x_1 - x_2) & \mbox{if}\,\,\, a_1+a_2 \equiv 1 \pmod{2}. \\
\end{array}
\right.
$$

An easy computation, on the lines of the preceding case, gives
\begin{lemma}
In the group $(\mathbb{Z}^2, \circ)$,  $0 x_1 + 0 x_2$ is the identity element and the inverse element of $a=a_1x_1+a_2x_2$ is given by
$$
\bar a =
\left\{\begin{array}{ll}
-a, & \mbox{if}\,\,\, a_1+a_2 \equiv 0 \pmod{2},  \\
  \\
  -a  +  (2 a_2 + p (a_1 + a_2)) (x_2 - x_1) & \mbox{if}\,\,\, a_1+a_2\equiv 1 \pmod{2}. \\
\end{array}
\right.
$$
\end{lemma}

\bigskip

If $a=a_1x_1+a_2x_2 \in \mathbb{Z}^2_0$, then
$$
a= \frac{a_1+a_2}{2}(x_1+x_2)+\frac{a_1-a_2}{2}(x_1-x_2),
$$
that is,  $\mathbb{Z}^2_0$ is generated by elements
$$
z_1 = x_1 + x_2,~~~z_2 = x_1 - x_2.
$$
Present  $\mathbb{Z}^2$ in the form
$$
\mathbb{Z}^2=\mathbb{Z}^2_{0} \bigsqcup(x_1+ \mathbb{Z}^2_{0}).
$$
Since $\varphi(\mathbb{Z}^2_{0})=\mathbb{Z}^2_{0}$, we get
$$
x_1+ \mathbb{Z}^2_{0}=x_1+\varphi(\mathbb{Z}^2_{0})=x_1\circ \mathbb{Z}^2_{0}.
$$
Hence
$$
(\mathbb{Z}^2, \circ) =\mathbb{Z}^2_{0} \sqcup(x_1\circ \mathbb{Z}^2_{0})
$$
is a decomposition of $(\mathbb{Z}^2, \circ )$ into cosets of  $\mathbb{Z}^2_{0}$.

To find the structure of $(\mathbb{Z}^2, \circ )$, we notice that this group is generated by the elements $x_1, z_1, z_2$ and
 the following relations hold:
$$
z_1 \circ z_2 = z_2 \circ z_1,~~~x_1 \circ x_1 = z_1 + (p+1) z_2.
$$
Since $\mathbb{Z}^2_{0}$ is normal in $(\mathbb{Z}^2, \circ )$, we get the following conjugation relations:
$$
\bar x_1 \circ z_1 \circ x_1 =  z_1 + 2(1 + p) z_2,~~~\bar x_1 \circ z_2 \circ x_1 = - z_2.
$$
Hence, we get
\begin{eqnarray*}
(\mathbb{Z}^2, \circ) &= &\langle x_1, z_1, z_2 \mid z_1 \circ z_2 = z_2 \circ z_1,~~x_1 \circ x_1 = z_1 \circ z_2^{\circ (p+1)},\\
& &  \;\; \bar x_1 \circ z_1 \circ x_1 =  z_1 + 2(1 + p) z_2, \bar x_1 \circ z_2 \circ x_1 = \bar z_2 \rangle.
\end{eqnarray*}
Notice that  $x_1^{\circ 2} \circ {\bar z_2}^{\circ (p+1)} = z_1$. Thus $(\mathbb{Z}^2, \circ)$ can finally be presented as
$$
(\mathbb{Z}^2, \circ) = \langle x_1, z_2 \mid \bar x_1 \circ z_2 \circ x_1 = \bar z_2  \rangle,
$$
which  is the Klein bottle group.

Let $(\mathbb{Z}^2, +, \circ)$ be a $\lambda$-cyclic brace. Then $\lambda : (\mathbb{Z}^2, \circ) \to \Aut (\mathbb{Z}^2, +)$ is a homomorphism, $\im \lambda \in \gen{\varphi}$ for some $\varphi \in \Aut (\mathbb{Z}^2, +)$ and $H_{\lambda}$ is a regular subgroup of $\Hol (\mathbb{Z}^2, +)$.  Then, by Proposition \ref{cycprop2}, $\varphi$ is of the form
$$\varphi(x_i) \equiv x_i \; (\mathrm{mod} \, \mathbb{Z}^n_0), \quad i=1, \ldots,n.$$
Hence by the preceding discussion, we get the desired characterization in the following

\begin{thm}
On the free abelian group of rank 2, the multiplicative group of any  $\lambda$-cyclic brace $(\mathbb{Z}^2, +, \circ)$ is one of the following:

1) $(\mathbb{Z}^2, \circ) = (\mathbb{Z}^2 , +)$.

2) $(\mathbb{Z}^2, \circ)$ is isomorphic to $(\mathbb{Z}^2 , +)$.

3)  $(\mathbb{Z}^2, \circ)$ is isomorphic to the Klein bottle group.
\end{thm}

\medskip

We again get back to the general case. 
Let $A := (\mathbb{Z}^n, +) = \gen{x_1, \ldots, x_n}$ be the free abelian group of rank $n$ and  $\psi \in \Aut A$  the cyclic permutation of free generators given by
$$
\psi :
\left\{
\begin{array}{ll}
x_i \to x_{i+1} &  ~\mbox{for} ~i = 1, 2, \ldots, n-1,\\
x_n \to x_1. &
\end{array}
\right.
$$
An arbitrary element of $A$ is of the form $a = \sum_{i=1}^n \alpha_i x_i$, where $\alpha_i \in \mathbb{Z}$.
Notice that $A$ is a homogeneous group and $\psi$ satisfies
$$
l(\psi (x_i)) = 1, \quad i = 1, \ldots, n.
$$
Hence, by Theorem \ref{lcb}, $H_{\lambda}$ is a regular subgroup of $\Hol A$, where $\lambda : A \to \Aut A$,  $a \mapsto \lambda_a$ with  $\lambda_a = \psi^{\sum\limits_{i=1}^{n} \alpha_i}$, is a homomorphism. Thus we get a $\lambda$-cyclic brace $(A, +, \circ)$, where `$\circ$' is defined by
$$
a \circ b = a+\psi^{\sum\limits_{i=1}^{n} \alpha_i}(b),
$$
for all
$a = \sum_{i=1}^n \alpha_i x_i, ~~~b = \sum_{i=1}^n \beta_i x_i \in A$.

We'll now find the structure of the multiplicative group $(A, \circ)$. Notice that
$$
x_1 \circ x_2=x_1+x_3 \;\;\mbox{ and } \;\; x_2 \circ x_1=2x_2.
$$
Thus $x_1 \circ x_2\neq x_2 \circ x_1$, and therefore $(A, \circ)$ is non-abelian.
Let
$$
A_0 =
\left\{ \, \sum\limits_{i=1}^{n} \alpha_i x_i \in A \mid
\sum\limits_{i=1}^{n} \alpha_i \equiv 0 \,\;(\mathrm{mod}\, n)  \, \right\}.
$$
Since $\psi^n=id$, the operations `$+$' and `$\circ$' coincide on $A_0$. It is not difficult to see that $A_0 = \Ker  \lambda$.

We have the following decomposition of  $A$ into the union of cosets of $A_0$:
$$
A =
A_0 \bigsqcup (x_1 + A_0) \bigsqcup \ldots \bigsqcup ((n-1) x_1 + A_0).
$$
Since $\psi(A_0)=A_0$, it follows that
$$
(kx_1)\circ A_0 =(kx_1) + A_0, \quad k\in \mathbb{Z}.
$$
Hence, we get the following  decomposition of   $(A, \circ)$  into the union of cosets  of $A_0$ :
$$
(A, \circ)= A_0 \bigsqcup (x_1 \circ A_0) \bigsqcup \ldots \bigsqcup (((n-1)x_1)\circ A_0).
$$
Moreover, notice that
$$
kx_1\equiv x_1^{\circ k} \;(\mathrm{mod}\, A_0), \quad k=1, 2, \ldots, n.
$$
Hence,
$$
(A_0, \circ) = A_0 \bigsqcup (x_1\circ A_0) \bigsqcup \ldots \bigsqcup (x_1^{\circ (n-1)} \circ A_0),
$$
and  therefore $(A, \circ)$ is generated by element  $x_1$ and the subgroup  $(A_0, \circ)$.

One can take
$$
\{z_1:=x_1+\cdots+x_n, \quad z_2:=x_2-x_1, \ldots,  z_n:=x_n-x_{n-1}\}
$$
as a set of free generators of  $(A_0, \circ)$. Let us denote $\bar a \circ \bar b \circ a \circ b$ by $[a, b]$ for all $a , b \in (A, \circ)$. We have relations
$$
x_1^{\circ n} =z_1, \quad [z_i, z_j]=1, \quad 1 \le i, j \le n,
$$
\begin{eqnarray*}
x_1\circ z_k &=&x_1\circ (x_k-x_{k-1})=x_1+ x_{k+1}-x_k =z_{k+1}+x_1\\
&=& z_{k+1}\circ x_1, \;\; 2 \le k \le n-1
\end{eqnarray*}
and
\begin{eqnarray*}
x_1\circ z_n &=& x_1\circ (x_n-x_{n-1})=x_1+x_1-x_n=(x_1-x_n) \circ x_1\\
&=& (\overline{z_2+z_3+\cdots+z_n}) \circ x_1= \bar z_2\circ \bar z_3\circ \cdots \circ \bar z_n \circ x_1.
\end{eqnarray*}
We have proved that  $(A, \circ)$ has a presentation
\begin{eqnarray*}
(A, \circ) &= & \left \langle \, x_1,z_1,\ldots,z_n \mid z_1=x_1^{\circ n},\, [z_i,z_j]=1 \,\; (1 \le i,j \le n), \; x_1 \circ z_k \circ \bar x_1=z_{k+1}
\right.\\
& &  \left. \; \; (2 \le k \le n-1),\, x_1 \circ z_n  \circ \bar x_1 = \bar z_2 \circ \bar z_3 \circ \cdots \circ \bar z_n\, \right\rangle.
\end{eqnarray*}

Let us simplify the presentation of $(A, \circ)$ for   $n\geq 2$ as follows:
\begin{eqnarray*}
(A, \circ) &=& \left\langle \, x_1,z_2,\ldots,z_n \mid [x_1^{\circ n},z_2]= \cdots  = [x_1^{\circ n},z_n]=1,\,\, [z_i,z_j]=1 \,\, (2 \le i,j \le n), \right.\\
& & \left. z_3=x_1 \circ z_2 \circ \bar x_1, \ldots, z_n=x_1 \circ z_{n-1} \circ \bar x_1,\,\,
x_1 \circ z_n \circ \bar x_1 = \bar z_2 \circ \bar z_3 \circ \cdots \bar \circ z_n \,\right\rangle\\
&=& \left\langle \, x_1,\,\,z_2 \mid
[x_1^{\circ n},z_2]=[x_1^{\circ n},x_1\circ z_2 \circ \bar x_1]= \cdots = [x_1^{\circ n},x_1^{\circ(n-2)} \circ z_2 \circ \bar x_1^{\circ (n-2)}]=1, \right.\\
& &
[z_2,x_1 \circ z_2 \circ \bar x_1]= \cdots = [z_2,x_1^{\circ (n-2)} \circ z_2 \circ \bar x_1^{\circ (n-2)}]=1, \;\ldots,\\
& &[x_1^{\circ(n-3)} \circ z_2 \circ \bar x_1^{\circ(n-3)},x_1^{\circ (n-2)}\circ z_2\circ \bar x_1^{\circ(n-2)}]=1,\\
& & \left. x_1^{\circ (n-1)}\circ z_n \circ \bar x_1^{\circ (n-1)}=\overline{z_2 \circ x_1\circ z_2 \circ \bar x_1 \circ \cdots \circ x_1^{\circ (n-2)}\circ z_2 \circ \bar x_1^{\circ (n-2)}}
\,\right\rangle
\end{eqnarray*}
Setting  $x = x_1$ and $z = z_2$, we get
\begin{eqnarray*}
(A, \circ) &=& \left\langle \, x,\, z \mid [x^{\circ n},z]=1,\, [z, x\circ z \circ \bar x]= \cdots =[z, x^{\circ (n-2)}\circ z \circ \bar x^{\circ (n-2)}]=1, \right.\\
& & \left. \bar x \circ z \circ  x= \overline{(z \circ x)^{\circ (n-1)} \circ \bar x^{\circ (n-1)}}
\,\right\rangle\\
&=&
\left\langle \, x,\, z \mid [x^{\circ n},z]=1,\, [z, x^{\circ k}\circ z \circ \bar x^{\circ k}]=1 \;\;(1 \le k \le n-2), \right.\\
& & \left. \bar x \circ z \circ x = x^{\circ (n-1)} \circ (\overline{z \circ x})^{\circ (n-1)} \,\right\rangle\\
& = &
\left\langle \, x,\,z \mid [x^{\circ n}, z]=1,\, [z, x^{\circ k}\circ z \circ \bar x^{\circ k}]=1\,\;(1 \le k \le n-2),\, (\overline{z \circ x})^{\circ n}=x^{\circ n} \,\right\rangle.
\end{eqnarray*}
Since
\begin{eqnarray*}
[z,x^{\circ k}\circ z\circ \bar x^{\circ k}] &=& x^{\circ k} \circ [\bar x^{\circ k}\circ z\circ x^{\circ k}, z]\circ \bar x^{\circ k} = x^{\circ k} \circ [z, [z, x^{\circ k}], z] \circ \bar x^{\circ k}\\
&=&x^{\circ k}\circ [z,x^{\circ k}, z] \circ \bar x^{\circ k},
\end{eqnarray*}
 finally we get the following presentation of $(A, \circ)$:
$$
(A, \circ) = \left\langle \, x,\,z \mid [x^{\circ n},z] = 1,\,(z \circ x)^{\circ n}=x^{\circ n},\,
[z,x^{\circ k}, z]=1\,(1 \le k \le n-2) \,\right\rangle.
$$

We have proved

\begin{thm}
Let $A \cong \mathbb{Z}^n$ be the free abelian group with free generators
 $x_1, x_2, \ldots,x_n$, $n\geq 2$, and  $\psi \in \Aut A$ be the  automorphism defined by cyclic permutation of free generators:
$$
\psi :
\left\{
\begin{array}{ll}
x_i \to x_{i+1} &  ~\mbox{for} ~i = 1, 2, \ldots, n-1,\\
x_n \to x_1. &
\end{array}
\right.
$$
Then $(A, +, \circ)$ is a $\lambda$-cyclic brace with
$$
a\circ b=a+\psi^{\sum\limits_{i=1}^{n} \alpha_i} (b), \quad a =\sum_{i=1}^n \alpha_i x_i, b \in A.
$$
Moreover,
$$
(A, \circ) = \left\langle \, x,\,z \mid [x^{\circ n},z] = 1,\,(z \circ x)^{\circ n}=x^{\circ n}, \,
[z,x^{\circ k}, z]=1 \,(1 \le k \le n-2) \,\right\rangle,
$$
where  $x=x_1$ and $z=x_2-x_1$.
\end{thm}

\bigskip


\section{Symmetric  skew braces} \label{sym}

We start with the definition of symmetric skew brace.

\begin{definition}
A skew brace $(G, \cdot, \circ)$ is said to {\it symmetric} if  $(G, \circ, \cdot)$ is also a skew brace.
\end{definition}

These braces were introduced in the paper of Childs \cite{Chi} (under the name bi-skew brace) and studied by Caranti \cite{Ca}. In this section we study symmetric skew braces.  We show that $\lambda$-cyclic skew braces and the skew braces constructed by exact factorization are symmetric.

\begin{proposition}\label{symprop1}
A skew brace  $(G, \cdot, \circ)$ is symmetric if and only if
$$
\overline{b}\circ (a\cdot b)\circ \overline{a} \in Ker\, \lambda
$$
for all $a,b \in G$, where $\bar a$ denotes the inverse of $a$ under `$\circ$'.
\end{proposition}
\begin{proof}
Let $(G, \cdot, \circ)$ be a skew brace. Then the map $\lambda : (G, \circ) \to \Aut (G, \cdot)$,  $a \mapsto \lambda_a$, is a homomorphism, where 
\begin{equation}\label{eqn1sec5}
\lambda_a(b)=a^{-1} \cdot (a\circ b)
\end{equation}
 for all $a, b \in G$. Now $( G, \circ, \cdot)$ is a skew brace if and only if 
$$
a\cdot(b\circ c)=(a\cdot b)\circ \overline{a}\circ (a\cdot c)
$$
for all $a, b, c \in G$.
 For the simplicity of notation, we'll suppress the use of `$\cdot$'.
Using \eqref{eqn1sec5}, we get
\begin{eqnarray*}
a (b\circ c)=(a b)\circ \overline{a}\circ (a c) & \Longleftrightarrow &
ab \lambda_b(c)=ab \lambda_{ab} (\overline{a}\lambda_{\overline{a}} (ac))\\
& \Longleftrightarrow &
\lambda_b(c)=\lambda_{ab} (\overline{a}\lambda_{\overline{a}} (a)\lambda_{\overline{a}} (c)).
\end{eqnarray*}
Since $\lambda_{\overline{a}} (a)={\overline{a}}^{-1}(\overline{a} \circ a)={\overline{a}}^{-1}$, we have
\begin{eqnarray*}
\lambda_b(c)=\lambda_{ab} (\lambda_{\overline{a}} (c)) &\Longleftrightarrow&
\lambda_b(c)=\lambda_{(ab)\circ \overline{a}} (c)\\
&\Longleftrightarrow&
\lambda_{\overline{b} \circ (ab)\circ \overline{a}} (c)=1.
\end{eqnarray*}
Since $a, b, c$ are the arbitrary elements of $G$, the assertion holds.
\end{proof}

\medskip

\begin{cor}\label{symcor1}
The equality $\lambda_{\overline{b} \circ (ab)\circ \overline{a}} =id$
is equivalent to any of the equalities
$$
\lambda_{\overline{a}\circ\overline{b} \circ (ab) }=id,\quad
\lambda_{b\circ a  }=\lambda_{ab } ,\quad
\lambda_{ (\overline{ab}) \circ (b \circ a)}=id.
$$
\end{cor}

\medskip
We first present some simple examples.
\begin{example}
Consider a   skew brace $( \mathbb{Z}, +, \circ)$,
where  `$+$'  is the addition of integers and  `$\circ$' is defined by the formula
$$
m \circ n=m+(-1)^{m}n, \quad   m,n \in \mathbb{Z}.
$$
In this case $\mathrm{Ker}\, \lambda=2\mathbb{Z}$.

Note that the following equality holds
$$
m\circ n\equiv m + n \, (\mathrm{mod}\, 2),~~~
\overline{m}= (-1)^{m+1}m \equiv m  \, (\mathrm{mod}\, 2), \quad   m,n \in \mathbb{Z}.
$$
Thus,
$$
\overline{m} \circ (m+n)\circ \overline{n} \equiv 2m+2n \equiv 0\, (\mathrm{mod}\, 2).
$$
Hence, by Proposition \ref{symprop1},  the algebraic system $(\mathbb{Z}, \circ, + )$ is a  skew brace.
\end{example}

\medskip

\begin{example}
Consider the skew brace $(\mathbb{Z}^n, +, \circ)$,
which is constructed on the  free abelian group  $\mathbb{Z}^n=\left\langle x_1,\ldots,x_n\right\rangle$ of rank $n$ using the automorphism   $\varphi\in \Aut \, \mathbb{Z}^n$ such that
$$
\varphi(x_i) \equiv x_i\, (\mathrm{mod}\, \mathbb{Z}^n_0), \quad i=1,\ldots,n
$$
and
$$
\mathbb{Z}^n_0=
\left\{\, \sum\limits_{i=1}^{n} a_ix_i \in  \mathbb{Z}^n \, | \, \sum a_i=0 \,\right\}=\mathrm{Ker}\, \lambda.
$$

Since
$$
a \circ a=a+\varphi^{\sum a_i}(b)\equiv a+b \, (\mathrm{mod}\, \mathbb{Z}^n_0)
$$
and
$$
\overline{a}=-\varphi^{-\sum a_i}(a)\equiv -a \, (\mathrm{mod}\, \mathbb{Z}^n_0),
$$
then
$$
\overline{b} \circ (a+b)\circ \overline{a} \equiv -b+a+b-a \equiv 0\, (\mathrm{mod}\, \mathbb{Z}^n_0).
$$
Hence, by Proposition \ref{symprop1}, the algebraic system $(\mathbb{Z}^n, \circ, +)$ is a  skew brace.
\end{example}

More examples can be obtained by
\begin{thm} \label{t6.6}
Every $\lambda$-cyclic skew brace is symmetric.
\end{thm}

\begin{proof}
Let   $(G, \cdot, \circ)$ be a $\lambda$-cyclic skew brace. Then $H_\lambda$ is a regular subgroup of $\Hol (G, \cdot)$ and $\lambda : (G, \cdot) \to \Aut (G, \cdot)$ is a homomorphism. Thus, by Theorem \ref{t1},  we have
$$
\lambda(b^{-1} \lambda_a(b))=id,
$$
which is equivalent to
$$
\lambda((ab)^{-1} (a\circ b))=id.
$$

Since the image of $\lambda$ is a cyclic subgroup of  $\Aut (G, \cdot)$, it follows that  $\Ker \lambda$ contains the commutator subgroup  of $(G, \cdot)$. Also, $(G, \cdot, \circ)$ being $\lambda$-homomorphic,  $\lambda$ is a homomorphism on $(G, \cdot)$ as well as on $(G, \circ)$.
Hence,
$$
\lambda((ab)^{-1} (a\circ b))=id \Longleftrightarrow \lambda((ba)^{-1} (a\circ b))=id  \Longleftrightarrow \lambda(ba) = \lambda(a\circ b). 
$$
Thus, by Proposition \ref{symprop1} and Corollary \ref{symcor1}, it follows that  $(G, \circ, \cdot)$ is a skew brace. The proof is complete.
\end{proof}

\medskip

Construction of skew braces on groups admitting exact factorization was carried out in \cite[Theorem 2.3]{SV}.
We here construct skew braces  by exact factorization on free product of groups. In particular, we do this on free groups.

Let $C$ and  $B$ be two groups, $G = C \ast B$ their free product and $D$ the kernel of the homomorphism  $C \ast B \to C \times B$.
Set $A=\gen{C, D}$. Then
$$
G=AB,\quad A\cap B=1,
$$
that is,  $G$ admits an exact factorization.

Now we can define the operation `$\circ$' on  $G$ by the rule: If $a_1b_1, a_2b_2 \in G$, where $a_1, a_2 \in A$, $b_1, b_2 \in B$,
then
\begin{equation}\label{efeqn1}
(a_1b_1)\circ (a_2b_2)=a_1a_2b_2b_1.
\end{equation}
It is not very difficult  to see that  $(G, \cdot, \circ)$  is a skew brace such that $(G, \circ)$ is isomorphic to the direct product  $A \times B_{op}$, where  $B_{op}$ is the opposite group of $B$.  If $B$ is abelian, then obviously $(G, \circ) \cong A \times B$. Notice that  `$\circ$' and   `$\cdot$' coincide on $A$.

 Let us  find automorphisms  $\lambda_z \in \mathrm{Aut} \, G$ corresponding to the elements
 $z \in (G, \circ)$. By the definition of $\lambda$, we have
$$
\lambda_z (y)=z^{-1}(z\circ y), \quad y \in G.
$$
If $z=a_1b_1$, $y=a_2b_2$, where  $a_1,a_2 \in A$, $b_1,b_2 \in B$,
then
$$
\lambda_z (y)=(a_1b_1)^{-1}a_1a_2b_2b_1=b_1^{-1} yb_1.
$$
So, the automorphism  $\lambda_z$ is  the inner automorphism of $G$ induced by the element  $b_1$, where $z = a_1b_1$.

We  now consider the map $\lambda: G \to \Aut(G)$ given by $z \mapsto \lambda_z$, where $\lambda_z$ is as defined in the preceding para. A straightforward computation shows that $\lambda$ is an anti-homomorphism. But, if $B$ is abelian, then $\lambda$ is a homomorphism. Hence we have shown

\begin{proposition}\label{efprop1}
The skew brace $(G, \cdot, \circ)$, constructed above, is   $\lambda$-homomorphic  if $B$ is abelian.
\end{proposition}

We next take up a specific free product.  Let  $F_n$ be the  free group  with basis $x_1,\ldots,x_n$. Let  $C = F_{n-1} = \gen{x_1, \ldots, x_{n-1}}$ and $B = \gen{x_n}$. Notice that $F_n = C \ast B$. Let
$A = \gen{C, [x_n, F_n]}$. Then
$$
F_n=AB,\quad A \cap B=1,
$$
that is, $F_n$ admits an exact factorization.

It then follows that $(F_n, \cdot, \circ)$ is a skew brace, where  `$\circ$' is defined by \eqref{efeqn1}.
Notice that  the group $(F_n, \circ)$  is isomorphic to the direct product  $A \times B$.
Also the operation `$\circ$' coincides with `$\cdot$' on $A$ as well as on  $B$.  Hence  $A$ is isomorphic to the free group  $F_{\infty}$ of infinite rank and  $B$  to the infinite cyclic group, that is,
$$
(F_n, \circ) \cong F_{\infty} \times \mathbb{Z}.
$$

We have thus proved
\begin{proposition}
There exists a skew brace $(G, \cdot, \circ)$ such that $(G, \cdot)$ is finitely generated, but $(G, \circ)$ is not.
\end{proposition}

Since $\mathbb{Z}$ is cyclic, by Proposition \ref{efprop1} it is easy to prove
\begin{proposition}
The skew brace $(F_n, \cdot, \circ)$ constructed by the exact factorization is $\lambda$-cyclic.
\end{proposition}

As an application of exact factorization we construct an example  of  skew brace whose additive group is  non-nilpotent solvable and  multiplicative group is abelian.

\begin{proposition}
There is a skew brace $(G, \cdot, \circ)$ such that $(G, \cdot) = \mathbb{Z} \wr \mathbb{Z}$ is a non-nilpotent metabelian group and $(G,  \circ)$ is a free abelian group of infinite rank.
\end{proposition}

\begin{proof}
Let a group $G$ have a presentation
$$
G=\left\langle\,  x, y_i, i\in \mathbb{Z}  \mid
   y_i^x=y_{i+1},\,\, [y_i,y_j]=1, \, i,j\in \mathbb{Z}\,\right\rangle.
$$
Notice that  $G$ is the direct wreath product  $\mathbb{Z} \wr \mathbb{Z}$ of two infinite cyclic groups.

Since $G/G'\cong \mathbb{Z}$, it follows that $G$ has exact factorization $G=G' \mathbb{Z}$. Hence, it is possible to define an operation `$\circ$'
such that the algebraic system $(G, \cdot, \circ)$ is a  skew brace and
$$
(G, \circ) \cong (G', \cdot) \times \mathbb{Z}.
$$
The commutator subgroup  $G'$ of $G$ is generated by elements  $y_i$, $i\in \mathbb{Z}$ and is isomorphic to the infinite direct product $\mathbb{Z}^{\infty}$.
Hence, $(G, \circ) \cong \mathbb{Z}^{\infty}$. Further, $\gamma_3\, G=\gamma_2\, G$ and $G''=1$. Thus the additive group
$(G, \cdot)$ is solvable but non-nilpotent.
\end{proof}

We next prove that skew braces constructed by exact factorization are symmetric.

\begin{thm}
Let $(G, \cdot, \circ)$ be a skew brace obtained by an exact factorization explained above.
Then  $(G, \cdot, \circ)$  is symmetric.
\end{thm}

\begin{proof}
Let $G = AB$ be an exact factorization. If  $x=ab \in G$, then $\bar x = a^{-1}b^{-1}$.
Hence, for $x_1=a_1b_1$, $x_2=a_2b_2 \in G$, we have
$$
\bar x_1 \circ \bar x_2 \circ (x_1x_2)=a_1^{-1}a_2^{-1}a_1b_1a_2b_1^{-1}.
$$
The automorphism  $\lambda_x$,  $x=ab \in G$, is equal to the inner automorphism induced by  $b$, that is,
$$
\lambda_x (y)=b^{-1}yb,\quad y \in G.
$$
So, if $x \in A$, then $\lambda_x = id$.
Since  $A$ is a normal subgroup of  $G$,  for all $a_1, a_2 \in A$ and $ b_1 \in B $, we have
$$
a_1^{-1}a_2^{-1}a_1b_1a_2b_1^{-1} \in A.
$$
Hence $\bar x_1 \circ \bar x_2 \circ (x_1x_2) \in \Ker \lambda$, and therefore, by  Proposition \ref{symprop1} and Corollary \ref{symcor1},   $(G, \circ, \cdot)$ is a skew brace, which completes the proof.
\end{proof}

\medskip

We conclude this section with another application of exact factorization. Let $K$ be a tame knot  in 3-sphere $\mathbb{S}^3$ and $G(K) = \pi_1(\mathbb{S}^3 \setminus K)$ be the knot group. Then we prove

\begin{proposition}
There is a non-trivial skew brace on $G(K)$.
\end{proposition}

\begin{proof}
Set $G = G(K)$. Since $G / G' \cong \mathbb{Z}$,  there exists  a factorization $G = A B$ with $A=G'$ and  $B = \langle x \rangle \cong \mathbb{Z}$, where $x$ is a meridian of the knot $K$. Then
$G = A B$ with $A \cap B = 1$. Hence, as explained above, we can define an operation $\circ$ on $G$ such that $(G, \cdot, \circ)$ is a non-trivial skew brace.
\end{proof}

\bigskip

\section{Some other constructions of skew braces} \label{const}

In this section we construct skew braces on arbitrary groups under various conditions. We start with a specific construction of two-sides braces.

It is well known that if $(A, +, \circ)$ is a finite brace, then the multiplicative group $(A, \circ)$ is solvable. Nasybullov \cite{Nas} proved that this is not true in general.  Moreover, a two sided brace was constructed in \cite{CSV19} whose multiplicative group  contains a non-abelian free group. Here we give a similar example.
Let us first recall construction of adjoint operation, which was used by Mal'cev \cite{Ma} and later by Sysak,  Rump and others (see \cite{A, R2007}).

Let $R$ be an associative ring, not necessarily with a unit element. The set of all elements of $R$ forms a semigroup $R^{ad}$, with neutral element $0 \in R$, under the operation
$$
a \circ b = a + b + ab~ \mbox{for all}~ a, b \in R,
 $$
which is called the {\it adjoint semigroup} of $R$. The group of all invertible elements of $R^{ad}$ is called the {\it adjoint group} of $R$ and is denoted by $R^0$. If $R$ has a unity 1, then $1 + R^{ad}$ coincides with the multiplicative semigroup $R^{mult}$ of $R$ and the map $r \mapsto 1 + r$ with $r \in R$ is an isomorphism from $R^{ad}$ onto $R^{mult}$.

A ring $R$ is called {\it radical} if $R = R^0$, which means that $R$ coincides with its Jacobson radical. Obviously such a ring has no unit element.

\begin{proposition} \label{p7.1}
There is a two-sided brace $(A, +, \circ)$ such that its multiplicative group $(A,  \circ)$ contains a non-abelian free subgroup.
\end{proposition}

\begin{proof}
Let $\mathbb{Z}\langle\langle X_1, X_2, \ldots, X_n \rangle\rangle$ be the ring of formal power series with non-commutative variables $X_1, X_2, \ldots, X_n$ over $\mathbb{Z}$. Let $B$ be the two-sided ideal of $\mathbb{Z}\langle\langle X_1, X_2, \ldots, X_n \rangle\rangle$, which is generated by $X_1, X_2, \ldots, X_n$. Define the following binary operation on $B$ the operation on $B$:
$$
a \circ b = ab + a + b,~~a, b \in B.
$$
We claim that $(B, +, \circ)$  is a brace. Indeed, it is evident that $(B, +)$ is an abelian group. Let us prove that $(B, \circ)$ is a group. Notice that the operation `$\circ$' is associative. Since
$$
0 \circ a = a \circ 0 = a,~~a \in B,
$$
$0$ is the identity element of $(B, \circ)$. For $a \in B$, we  are now going to find an element $y \in B$ such that $a \circ y = y \circ a = 0$. By the definition of `$\circ$', we have
$$
a \circ y = a y + a + y.
$$
If $a \circ y = 0$, then
$$
y = - a (1 + a)^{-1}  = - a + a^2 - a^3 + \cdots.
$$
Notice  that the element in the right side of the preceding equation  is well defined and lies in $B$. Hence $y$ is the right inverse of $a$. Also, it is easy to check that this element is the left inverse to $a$. Hence, $(B, \circ)$ is a group. A routine check now shows that $(B, +, \circ)$ is a two sided brace.

As we know there exists the Magnus embedding
$$
\mu : F_n \to \mathbb{Z}\langle\langle X_1, X_2, \ldots, X_n \rangle\rangle,
$$
which is defined on the generators and their inverses by the rule
$$
\mu(x_i) = 1 + X_i,~~~\mu(x_i^{-1}) = 1 - X_i + X_i^2 - \cdots.
$$
Then the group $\mu(F_n)$ is isomorphic to subgroup of $(B, \circ)$, which is generated by $X_1, X_2, \ldots, X_n$. Hence, $(B, \circ)$ contains a free non-abelian subgroup.
\end{proof}

A natural question is
\begin{qns}
Is it true that if the additive group of a brace is finitely generated, then the multiplicative group is solvable?
\end{qns}

\bigskip

We now construct braces on abelian groups. We prove

\begin{proposition}
Let $(A, + )$ be an abelian group,  $B$  its subgroup of index  2 and
$$
A=B \sqcup (a_0 + B)
$$
for some  $a_0 \in A$.  Then $(A, + , \circ)$ is a brace, where the operation `$\circ$' is dined by
$$
a \circ b =
\left\{
\begin{array}{ll}
a + b, \quad a \in B, \;\; b \in A,\\
a - b, \quad   a \not\in B,\;\; b \in A. &
\end{array}
\right.
$$
Moreover,  $(A, \circ)$ is isomorphic to the semi-direct product
$B \rtimes \mathbb{Z}_2$.
\end{proposition}

\begin{proof}
Let $( A, +)$ be an abelian group and  $\theta$
 an the inversion automorphism, that is,  $\theta (a)=-a$ for all $a \in A$.
Let us define a map $\lambda : (A, +) \to \Aut (A, +)$ by $a \mapsto \lambda_a$ such that
$$
\lambda_a =
\left\{
\begin{array}{ll}
id, \quad a \in B, \\
\theta, \quad  \  a \not\in B. &
\end{array}
\right.
$$
Notice that $\lambda$ is a homomorphism. A straightforward calculation shows that $b^{-1}\lambda_a(b) \in \Ker \lambda$ for all $a, b \in A$. Hence, by Theorem \ref{t1}, it follows that $H_{\lambda}$ is a normal subgroup of $\Hol (A, +)$. Hence $(A, + , \circ)$ is a brace, where `$\circ$' is as given in the statement.

Notice that $(A, \circ)$ is, in general, not an abelian group and $a_0 \circ a_0 = 0$, the identity element of $(A, +)$. Thus $\bar a_0 = a_0$, and therefore $\bar a_0 \circ b \circ a_0 = -b$ for all $b \in B$. Since `$\circ$' and `+' coincide on $B$, it follows that $(A, \circ) = B \rtimes \mathbb{Z}_2$.
\end{proof}

\medskip

A natural question which arises here is
\begin{qns}
Let $B_1$ and $B_2$ be subgroups of index  2 in $(A, +)$.
Under what conditions the braces constructed in the preceding result are isomorphic?
\end{qns}

It is not difficult to see that if $B_1$ and $B_2$ are not isomorphic, then the braces are not isomorphic. In particular, let $A = \mathbb{Z}_2 \times \mathbb{Z}_4$, where $\mathbb{Z}_2$ is generated by some element $a$ and  $\mathbb{Z}_4$ is generated by some element $b$. Put $B_1 = \langle b \rangle = \mathbb{Z}_4$ and $B_2 = \langle a \rangle \times \langle b^2 \rangle = \mathbb{Z}_2 \times \mathbb{Z}_2$. It is evident that these groups define non-isomorphic braces.

One may ask
\begin{qns} Is it possible to construct a skew brace whose additive group is an arbitrary group admitting  a subgroup of index 2?
\end{qns}

\noindent{\it Acknowledgements.}
The first and second named authors acknowledge the support from the Russian Science Foundation (project No. 19-41-02005). The third named author  acknowledges the support of DST-RSF Grant INT/RUS/RSF/P-2.

\end{document}